\documentclass{amsart}
\usepackage[utf8]{inputenc}
\usepackage{graphicx}
\usepackage{color}

\usepackage{amssymb}
\usepackage{mathrsfs,arydshln}
\newtheorem{thm}{Theorem}[section]
\newtheorem{cor}[thm]{Corollary}
\newtheorem{lem}[thm]{Lemma}

\theoremstyle{definition}
\newtheorem{defn}[thm]{Definition}
\newtheorem{conj}[thm]{Conjecture}
\theoremstyle{remark}

\numberwithin{equation}{section}

\newcommand{\N}{\mathbb{N}}

\newcommand{\ord}{{\rm ord~}}
\begin{document}

\title{On zero-sum subsequences in a finite abelian group of length not exceeding a given number}
\author[K. Zhao]{Kevin Zhao}
\address{School of Mathematics and Statistics, Nanning Normal University, Nanning 530100, China}
\address{Center for Applied Mathematics of  Guangxi, Nanning Normal University, Nanning 530100, China}
\email{zhkw-hebei@163.com}
\thanks{K. Zhao was supported partially by National Science Foundation
of China Grant \#12301425.}
\thanks{}%

\begin{abstract}
Let $G$ be an additive finite abelian group and let $k\in [\exp(G),\mathsf{D}(G)-1]$ be a positive integer.
Denote by $\mathsf{s}_{\leq k}(G)$ the smallest positive integer $l\in \mathbb{N}\cup \{+\infty\}$
such that each sequence of length $l$ over $G$ has a non-empty zero-sum
subsequence of length at most $k$.
Let $k_G\in [\exp(G),\mathsf{D}(G)-1]$ be the smallest positive integer such that $\mathsf{s}_{\leq \mathsf{D}(G)-d}(G)\leq \mathsf{D}(G)+d$ for $\mathsf{D}(G)-d\geq k_G$.
We conjecture that $k_G=\frac{\mathsf{D}(G)+1}{2}$ for finite abelian groups $G$ with $r(G)\geq 2$ and $\mathsf{D}(G)=\mathsf{D}^*(G)$.

In this paper, we mainly study this conjecture for finite abelian $p$-groups and get some results to support this conjecture. We also prove that $k_G\leq \mathsf{D}(G)-2$ for all finite abelian groups $G$ with $r(G)\geq 2$ except $C_2^3$ and $C_2^4$. In addition, we also get some lower bounds for the invariant  $\mathsf{s}_{\leq k}(G)$.
\end{abstract}
\maketitle
\section{Introduction}

Let $\mathbb{N}$ denote the set of positive integers, and
$\mathbb{N}_{0}=\mathbb{N}\cup\{0\}$. Throughout this paper all abelian
groups will be written additively, and for $n,r \in \mathbb{N}$, we let
$C_n$ denote the cyclic group of $n$ elements.
Let $G$ be an additive finite abelian group and $G^{*}=G\setminus\{0\}$.
It is well known that $|G|=1$ or $G = C_{n_1} \oplus C_{n_2} \oplus \cdots
\oplus C_{n_r}$, where $1 < n_1|n_2| \cdots |n_r$.
Then $r(G)=r$ is the rank of $G$ and
$\exp(G)=n_r $ is the exponent of $G$. Let
$$S:=g_1 \cdots g_l$$
be a sequence with elements in $G$.
We call $S$ a zero-sum sequence if  $g_1+ \cdots +g_l=0$ holds in $G$.
Let $\mathsf{D}(G)$ denote the  \textit{Davenport constant} of $G$, which is defined
as the smallest positive integer $t$ such that every sequence $S$ over
$G$ of length $|S| \geq t$ has a nonempty zero-sum subsequence. Set
$$\mathsf{D}^*(G):=1+\sum _{i=1} ^{r} (n_i-1).$$
Obviously, we have $\mathsf{D}(G)\geq \mathsf{D}^*(G)$.
In most of the cases, a direct zero-sum problem asks for the
the smallest integer $\ell\in \N$ such that  every sequence $S$ over $G$
with length $|S|\ge \ell$ has a zero-sum subsequence with prescribed length.

Let $L\subset \N$ be a nonempty subset and let $\mathsf s_{L}(G)$ be the
smallest $\ell\in \N\cup \{\infty\}$ such that every sequence $S$ over $G$
has a zero-sum subsequence $T$ with length $|T|\in L$. Thus for the classical
zero-sum invariants, we clearly have $\mathsf D(G)=\mathsf s_{\N}(G)$,
$\eta(G)=\mathsf s_{[1,\exp(G)]}(G)$, $\mathsf s(G)=\mathsf s_{\{\exp(G)\}}(G)$,  and $\mathsf{s}_{k \exp(G)}(G)=\mathsf{s}_{\{k \exp(G)\}}(G)$,
where $\mathsf s(G)$ is called \textit{Erd\"{o}s-Ginzburg-Ziv constant} of $G$ and $\mathsf{s}_{k \exp(G)}(G)$ is called \textit{Generalized Erd\"{o}s-Ginzburg-Ziv constant} of $G$.
The readers may want to consult one of the surveys or monographs
(\cite{GG06, Ge09a, GH06,GGS11, G13}). Moreover, $\mathsf s_L(G)$ is also
investigated for various other sets (see, e.g. \cite{Ch-Di-Ga-Ge-Sc, G01, [GHPS], G-L-P-W, Ku05}).

In this paper, we investigate the invarient $\mathsf s_L(G)$ for $L=[1,k]$.
\begin{defn}
Denote by $\mathsf{s}_{\leq k}(G)=\mathsf s_{[1,k]}(G)$ the smallest positive integer $l\in \mathbb{N}\cup \{+\infty\}$
such that each sequence of length $l$ has a non-empty zero-sum
subsequence of length at most $k$ $(k\in \mathbb{N})$.
\end{defn}
The constant $\mathsf{s}_{\leq k}(G)$ was introduced by Delorme,
Ordaz and Quiroz \cite{[DOQ]}.
It is trivial to see that $\mathsf{s}_{\leq k}(G)=\mathsf{D}(G)$ if $k\geq \mathsf{D}(G)$,
$\mathsf{s}_{\leq k}(G)=\eta(G)$ if $k=\exp(G)$ and
$\mathsf{s}_{\leq k}(G)=\infty$ if $1\leq k<\exp(G)$.

This invariant has very important research value.
Cohen and Zemor \cite{[CZ]} pointed out
a connection between $\mathsf{s}_{\leq k}(C_{2}^{r})$ and coding theory.
Bhowmik and  Schlagepuchta \cite{BS17} used $\mathsf{s}_{\leq k}(G)$ to study the Davenport constant $\mathsf{D}(G)$.
For more, one can see \cite{BS07,[DOQ]}.
Freeze and Schmid \cite{FS10} used $\mathsf{s}_{\leq k}(G)$ to study another invariant $\mathsf{D}_k(G)$, which is defined as the smallest integer $\ell\in \mathbb{N}$ such that every
sequence $S$ over $G$ of length $|S|\geq \ell$ has $k$ disjoint nonempty
zero-sum subsequences.
Recently, Gao et al. \cite{GHP22} studied the connection  between $\mathsf{s}_{\leq k \exp(G)}(G)$ and $\mathsf{s}_{k \exp(G)}(G)$, i.e., the following conjecture.

\begin{conj}\label{Gao}
Let $G$ be a finite abelian group. Then
$$\mathsf{s}_{k \exp(G)}(G) = \mathsf{s}_{\leq k \exp(G)}(G) + k \exp(G) - 1$$
for every $k\in \mathbb{N}$.
\end{conj}

Up to now, the direct and inverse zero-sum problems of $\mathsf{s}_{\leq k}(G)$ have been widely studied.
Freeze and Schmid \cite{FS10} obtained the exact number of $\mathsf{s}_{\leq 3}(C_{2}^{r})$, namely, $1+2^{r-1}$.
Lindstr\"{o}m \cite{L69} studied the invariant $\mathsf{s}_{\leq 4}(C_{2}^{r})$ when  $r$ goes to $+\infty$.
Wang  et al. \cite{WZ17} determined $\mathsf{s}_{\leq k}(G)$ for $G=C_m\oplus C_n$ with $m|n$.
Roy and Thangadurai \cite{RT18} considered the direct zero-sum problem of $\mathsf{s}_{\leq k}(G)$ for a finite abelian $p$-group $G$ with a large exponent by a different method from \cite{WZ17}.
Recently, Zhang \cite{Z23} continued to study this problem for abelian groups of rank three.

For the inverse zero-sum problem of $\mathsf{s}_{\leq k}(G)$, the known conclusions are less.
Davydov and Tombakis \cite{DT89} studied the inverse zero-sum problem of $\mathsf{s}_{\leq 3}(C_{2}^{r})$ by the method of coding theory.
Recently, the inverse zero-sum problem of $\mathsf{s}_{\leq k}(G)$ with $r(G)=2$ has been also solved.
For this, one can see \cite{EG22,DL22I,DL22II,GWZ20}.

Here, we only exhibit the result for the inverse zero-sum problem of $\mathsf{s}_{\leq k}(C_n\oplus C_n)$.

\begin{lem} \label{inv2}
Let $n\geq 2$, let $G = C_n\oplus C_n$, let $k\in [0, n-1]$, and let $S$
be a sequence of terms from $G$ with length $|S| = \mathsf{D}(G) + k-1 = 2n-2 + k$ having no nonempty
zero-sum subsequence of length at most $\mathsf{D}(G)-k = 2n-1-k$. Then there exists a basis $(e_1, e_2)$
for $G$ such that the following hold.
\begin{description}
  \item[(1)] If $k = 0$, then $Sg$ satisfies the description given in Item 2, where $g =-\sigma(S)$.
  \item[(2)] If $k = 1$, then
$$S = e_1^{n-1}\boldsymbol{\cdot}\prod_{i\in [1,n]}(x_ie_1 + e_2),$$
for some $x_1,\ldots , x_n\in [0, n-1]$ with $x_1 +\ldots + x_n\equiv  1 \mod n.$
  \item[(3)] If $k\in [2, n-2]$, then
$$S = e_1^{n-1}\boldsymbol{\cdot}e_2^{n-1}\boldsymbol{\cdot}(e_1+e_2)^{k}.$$
  \item[(4)]If $k = n-1$, then
$$S = e_1^{n-1}\boldsymbol{\cdot}e_2^{n-1}\boldsymbol{\cdot}(xe_1+e_2)^{k}$$
for some $x\in[1, n-1]$ with $\gcd (x, n) = 1$.
\end{description}
\end{lem}

We also collect some results for the direct zero-sum problem of $\mathsf{s}_{\leq k}(G)$.

\begin{lem} [\cite{WZ17}, Lemma 8]
Let $G$ be a finite abelian group with $r(G)\geq 2$.
Then $\mathsf{s}_{\leq \mathsf{D}(G)-1}(G)=\mathsf{D}(G)+1$.
\end{lem}

\begin{lem} [\cite{WZ17}] \label{r(G)=2}
Let $G$ be a finite abelian group with $r(G)=2$.
Then $\mathsf{s}_{\leq \mathsf{D}(G)-k}(G)=\mathsf{D}(G)+k$ for $\exp(G)\leq \mathsf{D}(G)-k\leq \mathsf{D}(G)$.
\end{lem}

\begin{lem}[\cite{BS17}] \label{colup}
$\mathsf{s}_{\leq \frac{3p-1}{2}}(C_p^d)\leq (6p-4)p^{d-3}+1$ for $d\geq 3$; $\mathsf{s}_{\leq 2p}(C_p^d)\leq (6p-4)p^{d-4}+1$ for $d\geq 4$; $\mathsf{s}_{\leq (d-1)p}(C_p^d)\leq (d-1)p$.
\end{lem}

\begin{lem} \label{col}
\begin{description}
  \item[(i)] $\mathsf{s}_{\leq 3}(C_3^3)=17$, $\mathsf{s}_{\leq 4}(C_3^3)=10$, $\mathsf{s}_{\leq 5}(C_3^3)=9$, $\mathsf{s}_{\leq 6}(C_3^3)=8$, $\mathsf{s}_{\leq 7}(C_3^3)=\mathsf{D}(C_3^3)=7. $ (\cite{BS07})
  \item[(ii)] $\mathsf{s}_{\leq 5}(C_5^3)=33$, $\mathsf{s}_{\leq 6}(C_5^3)=24$, $\mathsf{s}_{\leq 7}(C_5^3)=19$, $\mathsf{s}_{\leq 8}(C_5^3)=18$, $\mathsf{s}_{\leq 9}(C_5^3)=17$, $\mathsf{s}_{\leq 10}(C_5^3)=15$, $\mathsf{s}_{\leq 11}(C_5^3)=14$, $\mathsf{s}_{\leq 12}(C_5^3)=14$, $\mathsf{s}_{\leq 13}(C_5^3)=\mathsf{D}(C_5^3)=13.$ (\cite{S10})
  \item[(iii)] $\mathsf{s}_{\leq 2}(C_2^r)=2^{r}$; $\mathsf{s}_{\leq 3}(C_2^r)=2^{r-1}+1$; $\mathsf{s}_{\leq 4}(C_2^d)\leq 2^{(d+1)/2}+1$; $\mathsf{s}_{\leq r-k}(C_2^r)=r+2\ \text{for all}\
r-k\in \Big[\Big\lceil \frac{2r+2}{3}\Big\rceil, r\Big];$ $\mathsf{s}_{\leq 2m}(C_2^r)\leq (m+1)+(m!2^{r})^{1/m}$. (\cite{FS10,L69,WZ17})
  \item[(iv)] $\mathsf{s}_{\leq \mathsf{D}(G)-k}(G)=\mathsf{D}(G)+k$ for $r(G)=2$ and $\mathsf{D}(G)-k\in [\exp(G),\mathsf{D}(G)]$. (\cite{WZ17})
  \item[(v)] $\mathsf{s}_{\leq \mathsf{D}(G)-2}(G)=\mathsf{D}(G)+1$ for $G=C_p^r$ with $3\leq r<p$; $\mathsf{s}_{\leq \mathsf{D}(G)-p^n}(G)=\mathsf{D}(G)+p^n$ for $G=C_{p^n}^3$ with $p$ odd prime. (\cite{Z23})
\end{description}
\end{lem}

\begin{lem}[\cite{RT18}, Theorem 1.1] \label{colp}
Let $H$ be a finite abelian $p$-group with exponent $\exp(H) = p^m$ for some
integer $m\geq 1$ and for a prime number $p > 2r(H)$. Suppose
that $\mathsf{D}(H)-1 = k_0p^m + t$ for some integers $k_0\geq 1$
and $0\leq t\leq p^m-1$. Let $G = C_{p^n} \oplus H$ be a finite abelian $p$-group for some integer $n$
satisfying $p^n\geq 2(\mathsf{D}(H)-1)$. Let $k$ be any integer satisfying $\mathsf{D}(H)-1-k= ap^m +t'$ for some integer
$a$ satisfying $0\leq a\leq k_0-1$ and for some integer $t'$ satisfying $0\leq t'\leq t$. Then, we have
$$\mathsf{s}_{\leq \mathsf{D}(G)-k}(G)\leq\mathsf{D}(G)+k.$$
\end{lem}

From the above results, we can find that there may exist some $k_G\in [\exp(G),\mathsf{D}(G)-1]$ such that $\mathsf{s}_{\leq \mathsf{D}(G)-k}(G)\leq \mathsf{D}(G)+k$ for every finite abelian groups $G$, where $\mathsf{D}(G)-k\geq k_G$.
For this, we obtain the following:

We firstly consider the lower bounds of $\mathsf{s}_{\leq k}(G)$. The structures of sequences are similar with Lemma \ref{inv2}.

\begin{thm} \label{lowerCnr}
Let $G=C_n^r$ with $n,r\in \mathbb{N}_{\geq 2}$ and $k\in [0,n-1]$ be an integer.
Then $\mathsf{s}_{\leq 2n-1-k}(G)\geq 2^{r-1}(n-1)+1+k$.
\end{thm}

\textbf{Remark:} By Lemma \ref{col} Part (iii) and (iv), one can easily see that the lower bound is sharp for $n=2$ or $r=2$.

\begin{thm}\label{lower}
Let $G$ be a finite abelian group with $r(G)\geq 2$. If
$\exp(G)\leq \mathsf{D}^*(G)-k\leq 2\exp(G)-1$, then $\mathsf{s}_{\leq \mathsf{D}^*(G)-k}(G)\geq \mathsf{D}^*(G)+k$.

In particular, if $G$ is a finite abelian $p$-group with $r(G)\geq 2$ and $\mathsf{D}(G)\leq 2\exp(G)-1$, then $\mathsf{s}_{\leq \mathsf{D}(G)-k}(G)\geq \mathsf{D}(G)+k$  for $\exp(G)\leq \mathsf{D}(G)-k\leq \mathsf{D}(G)$.
\end{thm}

\textbf{Remark:} By the above result one can easily see that $\mathsf{s}_{\leq \mathsf{D}(G)-k}(G)=\mathsf{D}(G)+k$ in Lemma \ref{colp}.

For the finite abelian group $G$, we consider the case $k=2$.

\begin{thm}\label{D(G)-2}
Let $G$ be a finite abelian group with $r(G)\geq 2$. If $G\not\cong C_2^3$ or $C_2^4$ and $\mathsf{D}(G)-2\geq \exp(G)$,
then $\mathsf{s}_{\leq \mathsf{D}(G)-2}(G)\leq \mathsf{D}(G)+2$.

In particular, if $\mathsf{D}(G)=\mathsf{D}^*(G)$ and $\exp(G)\geq \frac{\mathsf{D}(G)-1}{2}$,
then $\mathsf{s}_{\leq \mathsf{D}(G)-2}(G)=\mathsf{D}(G)+2$.
\end{thm}

\textbf{Remark:} For $G=C_2^3$, we have that $\mathsf{D}(G)=4$ and $\mathsf{s}_{\leq 2}(G)=8$.
For $G=C_2^4$, we have that $\mathsf{D}(G)=5$ and $\mathsf{s}_{\leq 3}(G)=2^3+1=9$. It follows since $\mathsf{s}_{\leq 3}(C_2^r)=2^{r-1}+1$.

If $G$ is a finite abelian $p$-group, we can get the better results.

\begin{thm} \label{leqlp}
Let $G$ be a finite abelian $p$-group and  $k=c_1p^{t+1}+d\in [\exp(G)+1,\mathsf{D}(G)]$ with $d\in [0,p-1]$, $t\geq 0$ and $c_1\in [1,p-1]$.
Let $S$ be a sequence over $G$ of length $2\mathsf{D}(G)-k+1$.
Suppose that $2k-\mathsf{D}(G)\geq p+d-v$ and $2\mathsf{D}(G)-2k+1<\frac{p-1}{2}p^{t+1}$.
If $\binom{\mathsf{D}(G)}{k-1}\not\equiv 0 \mod p$, then $S$ has a nonempty zero-sum subsequence $T$ with $|T|\leq k-1$.
\end{thm}

\begin{thm}\label{leqG}
Let $G$ be a finite abelian $p$-group. Then $\mathsf{s}_{\leq k-1}(G)\leq 2\mathsf{D}(G)-k+1$ provided that one of the following holds
\begin{description}
    \item[(i)] $G=C_2^r$ and $k-1=\frac{r+2}{2}$ with $r=2^{t+1}-2$, $t\geq 1$;
    \item[(ii)] $G=C_p^4$ and $k-1=2p$ with $p\geq 5$;
    \item[(iii)] $G=C_p^d$ and $k-1=(d-1)p\in [p,\mathsf{D}(C_p^d)]$.
\end{description}
\end{thm}

\textbf{Remark:} It is easy to see that Theorem \ref{leqG} improves the results in Lemmas \ref{colup}, \ref{col}.

The paper is organized as follows. In Section $2$, we recall some basic notions, provide
several known results and also give the proofs of Theorems \ref{lowerCnr}, \ref{lower}. In Section $3$,
we present the proof for Theorem \ref{D(G)-2}. In Sections $4$, we consider a problem on zero-sum subsequences of zero-sum sequences over finite abelian $p$-groups, which is useful for the proofs of Theorems \ref{leqlp} and \ref{leqG}.  In Sections $5$, we present the proofs for Theorems \ref{leqlp} and \ref{leqG}.

\section{Preliminaries}


Throughout the paper, for real numbers $a\leq b$, we set
$[a, b] = \{x \in \mathbb{Z}: a\leq x\leq b\}$.

We consider the sequences as elements of the
free abelian monoid $\mathscr {F}(G)$ over $G$ and our notations and
terminologies coincide with \cite{GG06,GH06,G13}.
 Let
$$S=g_1 \boldsymbol{\cdot}\ldots\boldsymbol{\cdot} g_{\ell}= \prod_{g\in G}g^{\mathsf{v}_g(\alpha)}$$
be a sequence over $G$, where $\mathsf v_g(S)\in \N_0$ is the \textit{ multiplicity }
of $g$ in $S$ and $|S|=\ell=\sum_{g\in G}\mathsf{v}_g(S)\in \mathbb{N}_0$ is the \textit{length} of $S$.
The \textit{sum} of elements in $S$ is denoted by $\sigma(S)=\sum^{\ell}_{i=1}g_{i}
=\sum_{g\in G}v_g(S)\cdot g\in G$, and the maximal repetition of a
term in $S$ is denoted by $\mathsf{h}(S)$.
A sequence $S$ is called a \textit{zero-sum sequence} if $\sum^{\ell}_{i=1} g_i = 0\in G$.
We say $T$ is a \textit{subsequence} of $S$ if $\mathsf{v}_g(S) \ge \mathsf{v}_g(T)$
for all $g \in G$ and denote it by $T\mid S$. In particular, if $T$ is a subsequence
of $S$ and $T\neq S$, then we call $T$ the \textit{proper subsequence} of $S$.
If $T$ is \textit{empty}, then we say $T=1$. We say that $S$ is \textit{zero-free}
if it contains no nonempty zero-sum subsequence. If $S$ is a zero-sum sequence
and each proper subsequence is zero-free, then $S$ is called a \textit{minimal zero-sum sequence}.
By $\mathsf{supp}(S)$ we denote the subset of $G$ consisting of all elements
which $\mathsf v_g(S)>0$. We write $\langle S\rangle$ as a subgroup of $G$ generated by $\mathsf{supp}(S)$. If $T$ is a subsequence of $S$, we denote by $ST^{-1}$
the sequence obtained from $S$ by deleting $T$. If $T^n=\prod_{g|T}g^n$ is a subsequence of $S$ for some $n\in \mathbb{N}$, we denote by $ST^{-n}$
the sequence obtained from $S$ by deleting $T^n$.  We also write
$$\sum(S) = \{\sigma(T):T|S,T\neq 1 \}.$$

We collect some results on the groups $G$ with $\mathsf{D}(G)=\mathsf{D}^*(G)$.

\begin{lem}[\cite{GZZ15}]\label{D(G)}
Let $G$ be a finite abelain group.
Then $\mathsf{D}(G)=\mathsf{D}^*(G)$ holds for the following
\begin{enumerate}
\item[(a)] $r(G)\leq 2$;
\item[(b)] $G$ is a finite abelian $p$-group;
\item[(c)] $G= G' \oplus C_{k}$, where $G'$ is a finite abelian $p$-group with $\mathsf{D}(G')\leq 2\exp(G') - 1$ and $k$ is a positive integer not divisible by $p$;
\item[(d)] $G = C_2 \oplus C_{2m} \oplus C_{2n}$ with $m|n$;
\item[(e)] $G = C_3 \oplus C_{6m} \oplus C_{6n}$ with $m|n$;
\item[(f)] $G = C_{2p^a} \oplus C_{2p^b} \oplus C_{2p^c}$ with $a\leq b \leq c$ being nonnegative integers, and $p$ being a prime;
\item[(g)] $G = C^3_2 \oplus C_{2n}$.
\end{enumerate}

\end{lem}

\textit{The proof of Theorem \ref{lowerCnr}:}
By induction on $r$, we can show that there exists a sequence $S=T_0^{n-1}g_0^k$ over $C_n^r$ of length $2^{r-1}(n-1)+k$ such that each zero-sum subsequence of $S$ has length $\geq 2n-k$.
If $r=2$, then Lemma \ref{inv2} implies our result.
Assume that we have proved the result
for $<r$ and that $S'_1={g'}_1^{n-1}\ldots {g'}_m^{n-1}{g'}_0^k$ is the sequence over $C_n^{r-1}$ of length $2^{r-2}(n-1)+k$ satisfying that each zero-sum subsequence of $S'_1$ has length $\geq 2n-k$. Set $C_n^r=C_n^{r-1}\oplus \langle e_r\rangle$. Let $S_1=g_1^{n-1}\ldots g_m^{n-1}g_0^k$ is a sequence over $C_n^{r}$ by attaching $0$ to each element of $S'_1$ as the $r$-th entry. Obviously, each zero-sum subsequence of $S_1$ has length $\geq 2n-k$.
Let $$S=S_1(S_1g_0^{-k}+e_r)=g_0^k\prod_{i=1}^mg_i^{n-1}(g_i+e_r)^{n-1}.$$
It is easy to see that each nonempty zero-sum subsequence $T$ in $S$ must have the following form
$$T=g_0^{k_0}\prod_{i\in I_1}g_i^{v_i}\prod_{i\in I_2}(g_i+e_r)^{u_i},$$ where $k_0\in [0,k]$, $I_1,I_2$ are subsets of $[1,m]$ and $u_i,v_i\in [1,n-1]$.
It suffices to prove that $|T|\geq 2n-k.$

Obviously, $\sum_{i\in I_2}u_i\equiv 0 \mod n$.
In addition, $T':=g_0^{k_0}\prod_{i\in I_1}g_i^{v_i}\prod_{i\in I_2}(g_i)^{u_i}$ is zero-sum and has the form
$${T'}_1^nT'_2g_0^{k_0}$$ with $\mathsf{v}_g(T'_2g_0^{k_0})<n$ for any $g|T'_2$.
It is easy to see that $\sigma({T'}_1^n)=0$ and $T'_2g_0^{k_0}|S_1$.
Hence, $T'_2g_0^{k_0}$ is a zero-sum subsequence of $S_1$ of length $|T'_2g_0^{k_0}|\leq |T'|=|T|$.
Combining the structure of $S_1$ yields that either $|T|\geq|T'_2g_0^{k_0}|\geq 2n-k$ or $T'_2g_0^{k_0}$ is empty.
If the former holds, then we are done.
If the latter holds, then $k_0=0$ and $$T'=\prod_{i\in I_1}g_i^{v_i}\prod_{i\in I_2}(g_i)^{u_i}={T'}_1^n.$$
By $u_i,v_i\in [1,n-1]$ we must have that $I_1=I_2=I$ and $u_i+v_i=n$ for any $i\in I$.
Hence, $$T=\prod_{i\in I}g_i^{n-u_i}\prod_{i\in I}(g_i+e_r)^{u_i}$$  with $I\subset [1,m]$, $u_i\in [1,n-1]$, $\sum_{i\in I}u_i\equiv 0 \mod n$ and $|T|=|I|n$. If $|I|\geq 2$, then we are done. If $|I|=1$, then
$T=g_1^{n-u_1}(g_1+e_r)^{u_1}$ with $u_1\in [1,n-1]$ and $u_1\equiv 0 \mod n$. Obviously, this is impossible.
We complete the proof.
\qed

\textit{The proof of Theorem \ref{lower}:}
Set $G =H\oplus C_{n_r}=C_{n_1} \oplus \cdots \oplus C_{n_{r-1}}
\oplus C_{n_r}=\langle e_1\rangle \oplus \cdots \oplus \langle e_{r-1}\rangle
\oplus \langle e_r\rangle$, where $1 < n_1| \cdots|n_{r-1} |n_r$ and $\ord(e_i)=n_i$.
Let $$S=e_r^x\prod_{i=1}^{r-1}e_i^{n_i-1}(e_r-e_i)^{n_i-1}$$ be a sequence over $G$ of length $\mathsf{D}^*(G)+k-1$.
Since $\exp(G)\leq \mathsf{D}^*(G)-k\leq 2\exp(G)-1$, we have that $\mathsf{D}^*(H)\leq \exp(G)+k\leq \mathsf{D}^*(G)=\mathsf{D}^*(H)+\exp(G)-1$ and $x=|S|-\sum_{i=1}^{r-1}2(n_i-1)=\exp(G)+k-\mathsf{D}^*(H)\in [0,\exp(G)-1]$.
Obviously, each nonempty zero-sum subsequence $T$ of $S$ has the following form
$$T=e_r^{x_1}\prod_{i\in I}e_i^{v_i}(e_r-e_i)^{v_i},$$
where $x_1\in [0,x]$, $v_i\in [1,n_i-1]$ and $I\subset [1,r-1]$.
Since $\sigma(T)=(x_1+\sum_{i\in I}v_i)e_r=0$, we have that $n_r|(x_1+\sum_{i\in I}v_i)$, which implies that $|T|=x_1+\sum_{i\in I}2v_i=2(x_1+\sum_{i\in I}v_i)-x_1\geq 2n_r-x=2n_r-(n_r+k-\mathsf{D}^*(H))=\mathsf{D}^*(G)-k+1$.
Hence, $\mathsf{s}_{\leq \mathsf{D}^*(G)-k}(G)\geq |S|+1=\mathsf{D}^*(G)+k$ and we complete the proof.
\qed

\section{The proofs of Theorem \ref{D(G)-2}}

In this section, we prove Theorem \ref{D(G)-2}. In the proof, we use a technique from \cite{GYZ11}.

\begin{lem} [\cite{GH06}] \label{G}
Let $G$ be a finite abelian group and $S$ be a minimal zero-sum sequence over $G$ of length $\mathsf{D}(G)$. Then $\sum(S)=G$.
\end{lem}

\begin{lem} [\cite{GH06}]
Let $G$ be a cyclic group of order $n$ and $S$ be a sequence over $G$ of length $n$.  If $S$ is minimal zero-sum,
then $S=g^n$, where $g$ is a generator of $G$.
\end{lem}

By the above we can easily see that the following result holds.

\begin{cor} \label{cyc}
Let $G$ be a cyclic group of order $n$ and $S$ be a sequence over $G$ of length $\geq n$.  If every zero-sum subsequence of $S$ has the length $\geq n$,
then $S=g^{|S|}$, where $g$ is a generator of $G$.
\end{cor}

Let $G$ be a finite abelian group and  let $A$ be a finite non-empty subset of $G$. The stabilizer of $A$ is defined by $Stab (A) = \{g \in G : g + A = A\}$.

\begin{lem} [\cite{GH06}] \label{Stab}
Let $G$ be a finite abelian group and  let $A$ and $B$ be finite non-empty subsets of $G$. If $B\subset Stab(A)$, then $\langle B\rangle$ is a subgroup of $Stab(A)$.

Furthermore, the following statements are equivalent:
\begin{enumerate}
\item[(a)] $B\subset Stab(A)$.
\item[(b)] $A = A + B.$
\item[(c)] $A = (g_1 + \langle B\rangle) \cup \ldots\cup (g_r +\langle B\rangle)$ for some $r\in \mathbb{N}$ and $g_1,\ldots , g_r\in A.$
\end{enumerate}
\end{lem}

\begin{lem}\label{rank2}
Let $G$ be a finite abelian group with $r(G)\geq 2$ and $G$ does not have the form $C_2\oplus C_{2m}$. Let $S$ be a sequence over $G$ of length $\mathsf{D}(G)+1$. If every zero-sum subsequence of $S$ has the length $\geq \mathsf{D}(G)-1$,
then the rank of $\langle S\rangle$ is at least $2$.
\end{lem}

\begin{proof}
Suppose that the rank of $\langle S\rangle$ is $1$. Thus set $\langle S\rangle\cong C_n$. Since $r(G)\geq 2$, we have that $\langle S\rangle$ is a proper subgroup of $G$, which implies that $\mathsf{D}(G)\geq n+1$. It follows that every zero-sum subsequence of $S$ has the length $\geq \mathsf{D}(G)-1\geq n$.
By Corollary \ref{cyc} we have that $S=g^{\mathsf{D}(G)+1}$, where $g$ is a generator of $\langle S\rangle$. Obviously, $S$ has a zero-sum subsequence of length $n$. Hence, $n\geq \mathsf{D}(G)-1\geq n$, i.e., $\mathsf{D}(G)= n+1$. By a simple analysis, we must have that $G\cong C_2\oplus C_{n}$ with $2|n$, where $C_n=\langle S\rangle=\langle g\rangle$. This is a contradiction to our assumed condition.
\end{proof}

\begin{lem}\label{g2}
Let $G$ be a finite abelian group with $r(G)\geq 2$ and $G$ does not have the form $C_2\oplus C_{2m}$.
Let $S=Tg^2$ be a sequence over $G$ of length $\mathsf{D}(G)+1$, where $T$ is a zero-sum subsequence of length $\mathsf{D}(G)-1$.
Then $S$ has a zero-sum subsequence with length at most $\mathsf{D}(G)-2$.
\end{lem}

\begin{proof}
Suppose that every zero-sum subsequence of $S$ has the length $\geq \mathsf{D}(G)-1$. For any $g_1|S$ with $g_1\not\in \{g,2g\}$, we have that $\sigma(Sg_{1}^{-1})=\sigma(T)+2g-g_1=2g-g_1\neq 0$.
Since $|Sg_{1}^{-1}|=|S|-1=\mathsf{D}(G)$ and every zero-sum subsequence of $Sg_{1}^{-1}$ has the length $\geq \mathsf{D}(G)-1$, we must have that every zero-sum subsequence of $Sg_{1}^{-1}$ has the length $\mathsf{D}(G)-1$. It follows that $Sg_{1}^{-1}$ has an element $\sigma(Sg_{1}^{-1})=2g-g_1.$
Hence, $$g_1(2g-g_1)|S.$$

We claim that every zero-sum subsequence of $S':=S(g_1(2g-g_1))^{-1}g^2$ has the length $\geq \mathsf{D}(G)-1$ for any $g_1|S$ with $g_1\not\in \{g,2g\}$.
By $g_1\not\in \{g,2g\}$ we have that $g_1$ and $2g-g_1$ are not equal to $g$. Hence, $S'=S(g_1(2g-g_1))^{-1}g^2=T(g_1(2g-g_1))^{-1}g^4$. It follows that every zero-sum subsequence $T'$ of $S'$ has the form $T'=T'_1g^x$ with $T'_1|T(g_1(2g-g_1))^{-1}$ and $x\leq 4$. If $x\leq 2$, then $T'|S$, which implies that $|T'|\geq \mathsf{D}(G)-1$. If $x=3$ or $4$, then $T'=T'_1g^x=(T'_1g^2)g^{x-2}$ with $1\leq x-2\leq 2$ and $\sigma(T'_1g^2)=\sigma(T'_1g_1(2g-g_1))$. Obviously, $T'_1g_1(2g-g_1)g^{x-2}|Tg^2$ and $\sigma(T'_1g_1(2g-g_1)g^{x-2})=\sigma(T')=0$. Hence, $|T'_1g_1(2g-g_1)g^{x-2}|=|(T'_1g^2)g^{x-2}|=|T'|\geq \mathsf{D}(G)-1$.

For any $g_1|S$ with $g_1\not\in \{g,2g\}$, we may repeat the replacement procedure $S(g_1(2g-g_1))^{-1}g^2$.
Finally, we can obtain a sequence $S_1=g^a(2g)^b$ with $|S_1|=|S|=\mathsf{D}(G)+1$ satisfying that every zero-sum subsequence of $S_1$ has the length $\geq \mathsf{D}(G)-1$. Obviously, $r(\langle S_1\rangle)=1$, a contradiction to Lemma \ref{rank2}. We complete the proof.
\end{proof}

\begin{lem}\label{disg}
Let $G$ be a finite abelian group with $r(G)\geq 3$. Let $S=Tg_1g_2g_3$ be a sequence over $G$ of length $\mathsf{D}(G)+2$, where $g_1,g_2,g_3$ are pairwise distinct and $T$ is a zero-sum subsequence of length $\mathsf{D}(G)-1$.
If every zero-sum subsequence of $S$ has the length $\geq \mathsf{D}(G)-1$,
then $S$ must have the following form
\begin{equation}
\begin{aligned}
\label{S}
S&=\prod_{i=1}^{\ell}a_i(g_1+g_2-a_i)(g_1+g_2)^{\sigma}g_1g_2g_3 \\
&=\prod_{i=1}^{\ell}b_i(g_1+g_3-b_i)(g_1+g_3)^{\sigma}g_1g_2g_3 \\
&=\prod_{i=1}^{\ell}c_i(g_2+g_3-c_i)(g_2+g_3)^{\sigma}g_1g_2g_3
\end{aligned}
\end{equation}
satisfying that all elements of $S$ are pairwise distinct,
where $\sigma=0$ or $1$, $\ell=\frac{\mathsf{D}(G)-1-\sigma}{2}$ and $g_j+g_k\not\in \{g_i,2g_i\}$ for $\{i,j,k\}=\{1,2,3\}$.
\end{lem}

\begin{proof}
Since $r(G)\geq 3$, we have that $G$ does not have the form $C_2\oplus C_{2m}$. By the assumed condition, it is also easy to see that $T$ is minimal zero-sum.

Claim $1:$ Let $i,j\in\{1,2,3\}$ and $i\neq j$. For any $g|T$ with $g\not\in \{g_i,g_j,g_i+g_j\}$ we have that $g(g_i+g_j-g)|T$ and $g\neq g_i+g_j-g$.

Obviously, every zero-sum subsequence of $Tg^{-1}g_ig_j$ has the length $\geq \mathsf{D}(G)-1$, $|Tg^{-1}g_ig_j|=\mathsf{D}(G)$ and $\sigma(Tg^{-1}g_ig_j)=g_i+g_j-g\neq 0$. Thus every zero-sum subsequence of $Tg^{-1}g_ig_j$ must have the length $\mathsf{D}(G)-1$, which implies that $\sigma(Tg^{-1}g_ig_j)=g_i+g_j-g|Tg^{-1}g_ig_j$.
Since $g\not\in \{g_i,g_j,g_i+g_j\}$, we have that $g_i+g_j-g\not\in \{g_i,g_j\}$. Hence, $g_i+g_j-g|Tg^{-1}$, i.e., $g(g_i+g_j-g)|T$.

If $g=g_i+g_j-g$, then $S=Tg_1g_2g_3=T_1g(g_i+g_j-g)g_1g_2g_3=(T_1g_ig_jg^2)g_k$, where $\{i,j,k\}=\{1,2,3\}$. Obviously, $T_1g_ig_j$ is a zero-sum subsequence of $S$ of length $\mathsf{D}(G)-1$ and every zero-sum subsequence of $T_1g_ig_jg^2$ has the length $\geq \mathsf{D}(G)-1$. By Lemma \ref{g2} we can get a contradiction.
The claim is true.

Claim $2:$ $g_j+g_k\not\in \{g_i,2g_i\}$ for $\{i,j,k\}=\{1,2,3\}$.

W.o.l.g., let $i=1$, $j=2$ and $k=3$.
If $2g_1=g_2+g_3$, then consider the sequence $Tg_1^2$. Note that $T$ is minimal zero-sum. Obviously, any zero-sum subsequence $T'$ of $Tg_1^2$ has one of the following forms: (1) $T$; (2) $T_1g_1$, where $T_1|T$ with $\sigma(T_1)=-g_1$; (3) $T_2g_1^2$, where $T_2|T$ with $\sigma(T_2)=-2g_1$. If (1) holds, then $|T'|=|T|=\mathsf{D}(G)-1$. If (2) holds, then obviously $T_1g_1|S$, which implies that $|T'|=|T_1g_1|\geq \mathsf{D}(G)-1$. If (3) holds, then $2g_1=g_2+g_3$ implies that $\sigma(T_2g_1^2)=\sigma(T_2g_2g_3)=0$. It follows that $T_2g_2g_3$ is a zero-sum subsequence of $S$. Thus $|T'|=|T_2g_1^2|=|T_2g_2g_3|\geq \mathsf{D}(G)-1$.
Hence, every zero-sum subsequence of $Tg_1^2$ has the length $\geq \mathsf{D}(G)-1$. This is impossible by Lemma \ref{g2}.

If $g_1=g_2+g_3$, then there must exist $g|T$ with $g\not\in \{g_1,g_2,g_3\}$. Otherwise, $S=g_1^ag_2^bg_3^c$ with $a,b,c$ being positive integers and $T=g_1^{a-1}g_2^{b-1}g_3^{c-1}$.  If $a-1>0$, then since $g_1=g_2+g_3$, we have that $Tg_1^{-1}g_2g_3$ is a zero-sum subsequence of $S$ of length $|T|+1=\mathsf{D}(G)$. By the assumed condition, it is easy to see that $Tg_1^{-1}g_2g_3$ is minimal zero-sum. It follows from Lemma \ref{G} that $\sum(Tg_1^{-1}g_2g_3)=G$. Thus $r(\langle S\rangle)=r(G)\geq 3$. This is a contradiction to that $r(\langle S\rangle)=r(\langle g_1,g_2,g_3\rangle)=r(\langle g_2+g_3,g_2,g_3\rangle)\leq 2$.
If $b-1>0$ and $c-1>0$, then $T(g_2g_3)^{-1}g_1$ is a zero-sum subsequence of $S$ of length $|T|-1=\mathsf{D}(G)-2$.  This is a contradiction to the assumed condition.
If either $a-1=0$, $b-1=0$ or $a-1=0$, $c-1=0$, then $T=g_2^{b-1}$ or $g_3^{c-1}$ is a minimal zero-sum subsequence of length $\mathsf{D}(G)-1$. Thus $\mathsf{D}(G)-1=\ord(g_2)$ or $\ord(g_3)$. It follows that $G$ must have the form $C_2\oplus C_{2m}$ with $C_{2m}=\langle g_2\rangle$ or $\langle g_3\rangle$. This is contradiction to $r(G)\geq 3$.
Combining $g\not\in \{g_2+g_3,g_2,g_3\}$ with Claim $1$ yields that $g(g_2+g_3-g)|T.$ Thus $S=Tg_1g_2g_3=g(g_2+g_3-g)T'g_2g_3(g_2+g_3)$. Obviously, $T'(g_2+g_3)$  is a zero-sum subsequence of $S$ of length $|T|-1=\mathsf{D}(G)-2$.  This is a contradiction to the assumed condition.
The claim is true.

Claim $3:$ $g_1,g_2,g_3\nmid T$.

If $g_3|T$, then Claim $2$ implies that $g_1+g_2\neq g_3$. Since $g_1,g_2,g_3$ are pairwise distinct, we have that $g_3\not\in \{g_1,g_2,g_1+g_2\}$. It follows from Claim $1$ that $g_3(g_1+g_2-g_3)|T$.
Thus we can set $S=Tg_1g_2g_3=T_1g_3(g_1+g_2-g_3)g_1g_2g_3=(T_1g_1g_2g_3^2)(g_1+g_2-g_3)$. Obviously, $T_1g_1g_2$ is a zero-sum subsequence of $S$ of length $\mathsf{D}(G)-1$ and every zero-sum subsequence of $T_1g_1g_2g_3^2$ has the length $\geq \mathsf{D}(G)-1$. By Lemma \ref{g2} we can get a contradiction. Similarly, $g_1,g_2\nmid T$.

Claim $4:$ For any $g|T$ we have that $\mathsf{v}_g(T)=1$.

If there exists some $g|T$ with $\mathsf{v}_g(T)\geq 2$, then Claim $3$ implies that $g\not\in \{g_1,g_2,g_3\}$.
If $g\neq g_1+g_2$, then it follows from Claim $1$ that $g(g_1+g_2-g)|T$ and $g\neq g_1+g_2-g$. Thus we can set $S=Tg_1g_2g_3=g(g_1+g_2-g)T'g_1g_2g_3=(T'g_1g_2)g(g_1+g_2-g)g_3$ with $g|T'$. Obviously, $T'g_1g_2$  is a zero-sum subsequence of $S$ of length $\mathsf{D}(G)-1$. In addition, $g,g_1+g_2-g,g_3$ must be pairwise distinct. Otherwise, $S$ has a subsequence with the form $(T'g_1g_2)g_x^2$. This together with $r(G)\geq 3$ and Lemma \ref{g2} yields that $(T'g_1g_2)g_x^2$ has a zero-sum subsequence with length at most $\mathsf{D}(G)-2$, a contradiction to the lemma’s assumed condition. By Claim $3$ we have that $g,g_1+g_2-g,g_3\nmid T'g_1g_2$. Hence, $g
\nmid T'$, a contradiction. Hence, $\mathsf{v}_g(T)=1$ for any $g|T$ with $g\neq g_1+g_2$.
Similarly, $\mathsf{v}_g(T)=1$ for any $g|T$ with $g\not\in \{g_1+g_3,g_2+g_3\}$.
Since $g_1,g_2,g_3$ are pairwise distinct, we have that $g_1+g_2\not\in \{g_1+g_3,g_2+g_3\}$. Hence, $\mathsf{v}_{g_1+g_2}(T)=1$ if $g_1+g_2|T$.
The claim is true.

\

By the above claims, we can easily see that $S$ must have the following form
\begin{equation}
\begin{aligned}
\label{S1}
S&=\prod_{i=1}^{\ell_1}a_i(g_1+g_2-a_i)(g_1+g_2)^{\sigma_1}g_1g_2g_3 \\
&=\prod_{i=1}^{\ell_2}b_i(g_1+g_3-b_i)(g_1+g_3)^{\sigma_2}g_1g_2g_3 \\
&=\prod_{i=1}^{\ell_3}c_i(g_2+g_3-c_i)(g_2+g_3)^{\sigma_3}g_1g_2g_3
\end{aligned}
\end{equation}
satisfying that all elements of $S$ are pairwise distinct,
where each $\sigma_i=0$ or $1$.
Since $|S|=2\ell_1+\sigma_1+3=2\ell_2+\sigma_2+3=2\ell_3+\sigma_3+3=\mathsf{D}(G)+2$ and $\sigma_i=0$ or $1$, one can easily see that $\sigma:=\sigma_1=\sigma_2=\sigma_3=0$ or $1$ and $\ell:=\ell_1=\ell_2=\ell_3=\frac{\mathsf{D}(G)-1-\sigma}{2}$.
We complete the proof.
\end{proof}

\textit{The proof of Theorem \ref{D(G)-2}:}
If $r(G)=2$, then by Lemma \ref{col} Part (iv)  we are done. Hence, we can suppose $r(G)\geq 3$.
Note that $\mathsf{D}(G)-2\geq \exp(G)\geq 2$. It is easy to see that $\mathsf{D}(G)-2=2$ and $r(G)\geq 2$ iff $G\cong C_2^3$. Hence,  $\mathsf{D}(G)-2\geq 3$, i.e., $\mathsf{D}(G)\geq 5$. If $\mathsf{D}(G)=5$, then by a simple analysis, we have that either $r(G)\leq 2$ or $G\cong C_2^4$, a contradiction to $r(G)\geq 3$ and $G\not\cong C_2^4$. Hence, $$\mathsf{D}(G)\geq 6.$$

Let $S$ be a sequence over $G$ of length $\mathsf{D}(G)+2$.
Suppose that every zero-sum subsequence of $S$ has the length $\geq \mathsf{D}(G)-1$.
Since $\mathsf{s}_{\leq \mathsf{D}(G)-1}(G)=\mathsf{D}(G)+1<|S|$, there must exist a zero-sum subsequence in $S$ of length $\mathsf{D}(G)-1$. Thus we can set $$S=Tg_1g_2g_3,$$ where $T$ is a minimal zero-sum subsequence of $S$ of length $\mathsf{D}(G)-1$. By Lemma \ref{g2} we must have that $g_1,g_2,g_3$ are pairwise distinct.
It follows from Lemma \ref{disg} that $S$ must have the form (\ref{S}) with $\sigma=0$ or $1$ and $|S|=2\ell+\sigma+3=\mathsf{D}(G)+2\geq 8$, which implies that $\ell\geq 2$.

If $\sigma=1$, then since $g_1,g_2,g_3$ are pairwise distinct, we have that $g_1+g_2,g_1+g_3,g_2+g_3$ are also pairwise distinct. By (\ref{S}) we can suppose that $g_1+g_3=a_i, g_2+g_3=b_j, g_1+g_2=c_k$ for some $i,j,k\in [1,\ell]$, which implies that $T$ contains $g_1+g_2-a_i=g_2-g_3, g_1+g_3-b_j=g_1-g_2, g_2+g_3-c_k=g_3-g_1$. By Claim $2$ of Lemma \ref{disg} we have that $g_j+g_k\neq 2g_i$ for $\{i,j,k\}=\{1,2,3\}$.
Thus we must have that $g_1-g_2,g_2-g_3,g_3-g_1$ are pairwise distinct. Hence, $(g_1-g_2)(g_2-g_3)(g_3-g_1)$ is a zero-sum subsequence of $S$ of length $3\leq \mathsf{D}(G)-2$. This is a contradiction to the hypothesis.

If $\sigma=0$, then by (\ref{S}) $T$ can be viewed as a subset of $G$. In addition, we have that $g_j+g_k-g\in T$ and $g\neq g_j+g_k-g$ for any $g\in T$ and  $j,k\in\{1,2,3\}$ with $j\neq k$. Thus
$$\{g_1+g_2,g_1+g_3,g_2+g_3\}-T=T.$$
It follows that $T=g_{i_1}+g_{j_1}-T=g_{i_1}+g_{j_1}-(g_{i_2}+g_{j_2}-T)=(g_{i_1}+g_{j_1})-(g_{i_2}+g_{j_2})+T$ for any $\{i_1,j_1\},\{i_2,j_2\}\in \{\{1,2\},\{1,3\},\{2,3\}\}$. Obviously,
\begin{align*}
&\{(g_{i_1}+g_{j_1})-(g_{i_2}+g_{j_2}):\{i_1,j_1\},\{i_2,j_2\}\in \{\{1,2\},\{1,3\},\{2,3\}\}\}\\
&=\{g_i-g_j:i,j\in\{1,2,3\}\}.
\end{align*}
Hence, $$T=g_i-g_j+T \ \text{for any}\  i,j\in\{1,2,3\}.$$
Set $H=\langle g_1-g_2,g_1-g_3\rangle$.
It follows from Lemma \ref{Stab} that
$$T=T+H=\bigcup_{i=1}^{t'}(w_i+g_1+H).$$
Note that $g_1+H=g_2+H=g_3+H$.
For any $h\in H$ and $i\in [1,t']$, we have that
\begin{align*}
&g_j+g_k-(w_i+g_1+h)=-w_i+g_j+(g_k-g_1-h)\\
&\in -w_i+g_j+(g_k-g_1+H)=-w_i+g_j+H=-w_i+g_1+H,
\end{align*}
where $j,k\in\{1,2,3\}$ with $j\neq k$.
Since $g_j+g_k-g\in T$ for any $g\in T$ and  $j,k\in\{1,2,3\}$ with $j\neq k$, we can set
\begin{equation*}
\begin{aligned}
T=\bigcup_{x=1}^{t}\{v_{x}+g_1+H,-v_{x}+g_1+H\}\bigcup_{y=1}^{s}(u_{y}+g_1+H),
\end{aligned}
\end{equation*}
where each $2v_x\not\in H$ and $2u_y\in H$.
It is easy to see that
$$|T|=2t|H|+s|H|=\mathsf{D}(G)-1,$$
and for any $h\in H$ and $i,j\in\{1,2,3\}$ with $i\neq j$
\begin{align*}
&g_i+g_j-(v_x+g_1+h)\in -v_x+g_1+H,\\
&g_i+g_j-(u_y+g_1+h)\in u_y+g_1+H.
\end{align*}
In addition, if  $s>0$, then since $g\neq g_j+g_k-g$ for any $g\in T$, we must have that $|H|$ is even.
Denote by $H_1$ the set of elements of order $\leq 2$ in $H$. Thus we can write $$H=H_1\uplus A\uplus (-A).$$
Since $r(H)=r(\langle g_1-g_2,g_1-g_3\rangle)\leq 2$, we must have that $H_1\cong C_2$ or $C_2^2$ if $|H_1|>1$.
Since each $2u_y\in H$, there must exist some $e_y\in G$ of order $\leq 2$ such that $u_y+H=e_y+H$.
Thus
\begin{equation}
\label{T}
\begin{aligned}
T=&\bigcup_{x=1}^{t}\{v_{x}+g_1+H,-v_{x}+g_1+H\}\bigcup_{y=1}^{s}(u_{y}+g_1+H)\\
=&\bigcup_{x=1}^{t}\{v_{x}+g_1+H,-v_{x}+g_1+H\}\bigcup_{y=1}^{s}(u_{y}+g_1+(H_1\uplus A\uplus (-A)))\\
=&\bigcup_{x=1}^{t}\bigcup_{h\in H}\{v_{x}+g_1+h,-v_{x}+g_1-h\} \\
&\bigcup_{y=1}^{s}\bigcup_{h\in A}\{e_y+g_1+h,e_y+g_1-h\}\bigcup_{y=1}^{s}\bigcup_{h\in H_1}(e_y+g_1+h).
\end{aligned}
\end{equation}
Obviously, $\sum_{h\in H_1}(e_y+g_1+h)=2g_1+e$ if $H_1=\langle e\rangle\cong C_2$; $\sum_{h\in H_1}(e_y+g_1+h)=4g_1$ if $H_1\cong C_2^2$.
Note that $T$ is minimal zero-sum. Thus, it follows from (\ref{T}) that
\begin{equation}
\label{T'}
\begin{aligned}
T':&=\prod_{x=1}^{t}(2g_1)^{|H|}\prod_{y=1}^{s}(2g_1)^{|A|}\prod_{y=1}^{s}\sum_{h\in H_1}(e_y+g_1+h)\\
&=\left\{
   \begin{aligned}
&(2g_1)^{t|H|}, \ &&\text{ if $s=0$}\,,\\
   &(2g_1)^{t|H|+s|A|}(2g_1+e)^s, \ &&\text{ if $s> 0$ and $H_1=\langle e\rangle\cong C_2$}\,,\\
  &(2g_1)^{t|H|+s|A|}(4g_1)^s, &&\text{ if $s> 0$ and $H_1\cong C_2^2$}\,,
   \end{aligned}
   \right.
\end{aligned}
\end{equation}
is also minimal zero-sum.

In the following, we will distinguish two cases to complete our proof.

Case $1:$ $t|H|+s|A|>0$.

If $H_1=\langle e\rangle\cong C_2$ and $2\nmid s$, then $|A|=\frac{|H|-|H_1|}{2}=\frac{|H|-2}{2}$ and $\sigma(T')=(t|H|+s|A|+s)(2g_1)+e=(t+\frac{s}{2})|H|(2g_1)+e=0$, i.e., $e=(t+\frac{s}{2})|H|(2g_1)=(2t+s)|H|g_1\in \langle 2g_1\rangle$.
This together with $2t|H|+s|H|=\mathsf{D}(G)-1$ yields that $2\Bigm|\ord(2g_1) \Bigm| 2(t+\frac{s}{2})|H|=\mathsf{D}(G)-1$, which implies that $\ord(g_1)=2\ord(2g_1)$.
If $\ord(2g_1)\geq \frac{\mathsf{D}(G)-1}{2}$, then combining $r(G)\geq 3$ yields that $\exp(G)\geq\ord(g_1)=2\ord(2g_1)\geq\mathsf{D}(G)-1\geq \exp(G)+1.$ This is a contradiction.
Thus we can suppose that $\ord(2g_1)\leq\frac{\mathsf{D}(G)-3}{2}$ (since $2\mid \mathsf{D}(G)-1$).
By (\ref{T'}) we have that
$$T'':=(2g_1)^{t|H|+s|A|}(2g_1+e)^{s-1}=(2g_1)^{\frac{\mathsf{D}(G)-1}{2}-s}(2g_1+e)^{s-1}$$ is zero-free with $2|s-1$. Since  $t|H|+s|A|>0$, we have that $$\sum(T'')\supset \{i(2g_1):i\in [1,\frac{\mathsf{D}(G)-3}{2}]\}.$$
It follows that $\ord(2g_1)>\frac{\mathsf{D}(G)-3}{2}$. This is a contradiction to $\ord(2g_1)\leq\frac{\mathsf{D}(G)-3}{2}$.

If $s=0$ or $H_1=\langle e\rangle\cong C_2$ with $2|s$  or $H_1\cong C_2^2$, then it follows from (\ref{T'}) that
\begin{equation}
\label{T''}
\begin{aligned}
T'':=\left\{
   \begin{aligned}
&(2g_1)^{t|H|+s|A|}, \ &&\text{ if $s=0$}\,,\\
   &(2g_1)^{t|H|+s|A|}(4g_1)^{\frac{s}{2}}, \ &&\text{ if $s> 0$ and $H_1=\langle e\rangle\cong C_2$ with $2|s$}\,,\\
  &(2g_1)^{t|H|+s|A|}(4g_1)^s, &&\text{ if $s> 0$ and $H_1\cong C_2^2$}\,,
   \end{aligned}
   \right.
\end{aligned}
\end{equation}
is minimal zero-sum.
Since $t|H|+s|A|>0$ and $2t|H|+s|H|=\mathsf{D}(G)-1$, by a simple caculation we have that $$\sum(T'')=\{i(2g_1):i\in [1,\frac{\mathsf{D}(G)-1}{2}]\}.$$ Combining the minimality of $T''$ yields that $$\ord(2g_1)=\frac{\mathsf{D}(G)-1}{2}.$$
It follows that $$\ord(g_1)=\mathsf{D}(G)-1\ \text{or}\ \frac{\mathsf{D}(G)-1}{2}.$$
If $\ord(g_1)=\mathsf{D}(G)-1$, then we must have that $G$ has the form $C_2\oplus C_{2m}$, a contradiction to $r(G)\geq 3$.
If $\ord(g_1)=\frac{\mathsf{D}(G)-1}{2}$, then $\ord(g_1)=\ord(2g_1)$, which implies that $\frac{\mathsf{D}(G)-1}{2}$ must be odd.
It follows that $\sigma((2g_1)^{\frac{\mathsf{D}(G)-3}{4}}g_1)=0$.
Since $\sum(T)\supset\sum(T'')=\{i(2g_1):i\in [1,\frac{\mathsf{D}(G)-1}{2}]\}$, by the structure (\ref{T}) of $T$, there must exist a subsequence $T_1|T$ of length $\frac{\mathsf{D}(G)-3}{2}$ such that $\sigma(T_1)=\sigma((2g_1)^{\frac{\mathsf{D}(G)-3}{4}})=\frac{\mathsf{D}(G)-3}{2}g_1$. It follows from $S=Tg_1g_2g_3$ that $T_1g_1$ is a zero-sum subsequence of $S$ with $|T_1g_1|=\frac{\mathsf{D}(G)-3}{2}+1<\mathsf{D}(G)-1$. This is a contradiction to the
hypothesis of $S$.

Case $2:$ $t|H|+s|A|=0$.

If $s=0$, then by imitating the proof of Case 1 we can get a contradiction.
Hence, we must have that $s>0$ and $t=|A|=0$, which implies that $H=H_1$ and $|T|=2t|H|+s|H|=s|H_1|=\mathsf{D}(G)-1$.
If $H_1=\langle e\rangle\cong C_2$, then $H=\langle g_1-g_2,g_1-g_3\rangle=\langle e\rangle$. Since $g_1,g_2,g_3$ are pairwise distinct, it follows that $g_1-g_2,g_1-g_3$ are distinct and both nonzero. This is impossible.
Hence, $H=H_1\cong C_2^2$ and $s=\frac{\mathsf{D}(G)-1}{4}$. It follows from \ref{T'} that
$$T'=(4g_1)^{\frac{\mathsf{D}(G)-1}{4}}$$ is minimal zero-sum, i.e., $\ord (4g_1)=\frac{\mathsf{D}(G)-1}{4}=s$.
It is easy to see that $$\ord (4g_1)=\frac{\mathsf{D}(G)-1}{4}=s\ \text{must be odd},$$ otherwise, $\ord(g_1)=2\ord(2g_1)=4\ord(4g_1)=\mathsf{D}(G)-1$, a contradiction to $r(G)\geq 3$.
Since $H=\langle g_1-g_2,g_1-g_3\rangle\cong C_2^2$, and $g_1-g_2,g_1-g_3$ are distinct and both nonzero, we have that $g_1-g_2=e_1$ and $g_1-g_3=e_2$ are of order $2$ and constitute a basis of $H$.
It follows from (\ref{T}) that
$$S=g_1g_2g_3T=g_1(g_1+e_1)(g_1+e_2)\prod_{y=1}^{s}(e_y+g_1)(e_y+g_1+e_1)(e_y+g_1+e_2)(e_y+g_1+e_1+e_2)$$ with $s=\frac{\mathsf{D}(G)-1}{4}\geq 2$ odd (since $\mathsf{D}(G)\geq 6$).
Obviously, $\ord (g_1)=\ord (4g_1)=\frac{\mathsf{D}(G)-1}{4}$ or $\ord (g_1)=2\ord (4g_1)=\frac{\mathsf{D}(G)-1}{2}$ or $\ord (g_1)=4\ord (4g_1)=\mathsf{D}(G)-1$.
Since $r(G)\geq 3$, we have that $\ord (g_1)\neq\mathsf{D}(G)-1$.
Hence, $$\ord (g_1)=\frac{\mathsf{D}(G)-1}{4}\ \text{or} \ \frac{\mathsf{D}(G)-1}{2}.$$

If $\ord (g_1)=\frac{\mathsf{D}(G)-1}{2}$, then since $s\geq 2$ is odd, we have that $0\leq \frac{s-3}{2}<s-1$ and
\begin{align*}
&(g_1+e_1)(g_1+e_2)(e_s+g_1)(e_s+g_1+e_1)(e_{s-1}+g_1)(e_{s-1}+g_1+e_2)\\
&\prod_{y=1}^{\frac{s-3}{2}}(e_y+g_1)(e_y+g_1+e_1)(e_y+g_1+e_2)(e_y+g_1+e_1+e_2)
\end{align*}
is a zero-sum subsequence of $S$ of length $2s=\frac{\mathsf{D}(G)-1}{2}<\mathsf{D}(G)-1$. This is a contradiction to the
hypothesis of $S$.

If $\ord (g_1)=\frac{\mathsf{D}(G)-1}{4}$ and $s=\frac{\mathsf{D}(G)-1}{4}\equiv 1 \mod 4$, then $s\geq 5$ and
\begin{align*}
&g_1(g_1+e_1)(g_1+e_2)(e_s+g_1+e_1)(e_s+g_1+e_2)\\
&\prod_{y=1}^{\frac{s-5}{4}}(e_y+g_1)(e_y+g_1+e_1)(e_y+g_1+e_2)(e_y+g_1+e_1+e_2)
\end{align*}
is a zero-sum subsequence of $S$ of length $s=\frac{\mathsf{D}(G)-1}{4}<\mathsf{D}(G)-1$. This is a contradiction to the
hypothesis of $S$.

If $\ord (g_1)=\frac{\mathsf{D}(G)-1}{4}$ and $s=\frac{\mathsf{D}(G)-1}{4}\equiv 3 \mod 4$, then $s\geq 3$ and
$$(g_1+e_1)(e_s+g_1)(e_s+g_1+e_1)
\prod_{y=1}^{\frac{s-3}{4}}(e_y+g_1)(e_y+g_1+e_1)(e_y+g_1+e_2)(e_y+g_1+e_1+e_2)$$
is a zero-sum subsequence of $S$ of length $s=\frac{\mathsf{D}(G)-1}{4}<\mathsf{D}(G)-1$. This is a contradiction to the
hypothesis of $S$.
We complete the proof of the first assertion.

The second assertion immediately follows from Theorem \ref{lower} and \ref{D(G)-2}.
\qed

\section{Zero-sum subsequences of zero-sum sequences}

In \cite{FWe11},  Gao et.al. considered a problem on short zero-sum subsequences of zero-sum sequences over finite abelian groups. For more, one can see \cite{L17}. In this section, we study a problem on zero-sum subsequences of length $\leq k$ in zero-sum sequences over finite abelian $p$-groups. This problem is useful for the determination of $\mathsf{s}_{\leq k}(G)$ in next section.

We will use the following lemmas repeatly.
\begin{lem}[Lucas’ Theorem, \cite{L78}] \label{luc}
Let $a$, $b$ be positive integers with $a = a_np^n +\ldots + a_1p + a_0$ and $b = b_np^n +\ldots + b_1p + b_0$ be the $p$-adic expansions, where $p$ is a prime. Then
$$\binom{a}{b}\equiv \binom{a_n}{b_n}\binom{a_{n-1}}{b_{n-1}}\cdots\binom{a_0}{b_0}\ \mod \ p.$$
\end{lem}

For any integer $k\in \mathbb{Z}_{>0}$, $g\in G$, and a sequence $S$ over $G$, we define
$$\mathsf{N}_{g}^{k}(S) =\big |\big\{I\subset [1, |S|] :\sum_{i\in I}
g_i = g, |I| = k\big\}\big |,$$
which denotes the number of subsequences $T$ of $S$ having sum $\sigma(T) = g$ and length
$|T| = k$ (counted with the multiplicity of their appearance in $S$).
Denote by $\mathsf{N}_{g}^{+}(S)$ ($\mathsf{N}_{g}^{-}(S)$) the number of subsequences $T$ of $S$ of even (odd) length having sum $\sigma(T) = g$ (each counted with
the multiplicity of its appearance in $S$).
If $g=0$, we also write $\mathsf{N}_{0}^{k}(S)$ as $\mathsf{N}^{k}(S)$.

\begin{lem}[\cite{GH06}, Proposition 5.5.8] \label{lem1}
Let $p$ be a prime, let $G$ be an abelian $p$-group, and let  $S=g_1 \boldsymbol{\cdot}\ldots\boldsymbol{\cdot} g_\ell \in \mathscr {F}(G)$.
If $\ell\geq \mathsf{D}(G)$, then $\mathsf{N}_{g}^{+}(S)\equiv \mathsf{N}_{g}^{-}(S) \mod p$
for all $g\in G$. In particular, $\mathsf{N}_{0}^{+}(S)\equiv \mathsf{N}_{0}^{-}(S) \mod p$.
\end{lem}

\begin{lem} \label{matrix0}
Let $x$ be an integer and $c,k,u_1,u_2,\lambda$ be positive integers. Let
\begin{equation*}
A=\left(
  \begin{array}{cccccc}
    1+x & 1 & 1 & \cdots & 1\\
    {{c+u_1 \choose 1}} & {{c+u_2 \choose 1}} & {{c+k \choose 1}} & \cdots & {{c \choose 1}}\\
    {\vdots} & {\vdots} & {\vdots} & {\vdots} & {\vdots}\\
    {{c+u_1 \choose \lambda}} & {{c+u_2 \choose \lambda}} & {{c+k \choose \lambda}} & \cdots & {{c \choose \lambda}}\\
  \end{array}
\right)_{(\lambda+1)\times (k+3)}.
\end{equation*}
Then by some row transformations, $A$ can be brought to the following form
\begin{equation*}
A'=\left(
  \begin{array}{ccccccc}
    {{u_1 \choose 0}}+(-1)^0x\times {{c \choose 0}} & {{u_2 \choose 0}} & {{k \choose 0}} & \cdots & {{0 \choose 0}}\\
    {{u_1 \choose 1}}+(-1)^1x\times {{c \choose 1}} & {{u_2 \choose 1}} & {{k \choose 1}} & \cdots & {{0 \choose 1}}\\
    {{u_1 \choose 2}}+(-1)^2x\times {{c+1 \choose 2}}& {{u_2 \choose 2}} & {{k \choose 2}} & \cdots & {{0 \choose 2}}\\
    {\vdots} & {\vdots} & {\vdots} & {\vdots} & {\vdots}\\
    {{u_1 \choose \lambda}}+(-1)^{\lambda}x\times {{\lambda+c-1 \choose \lambda}} & {{u_2 \choose \lambda}} & {{k \choose \lambda}} & \cdots & {{0 \choose \lambda}}\\
  \end{array}
\right).
\end{equation*}
\end{lem}

\begin{proof}
We take some row transformations of $A$ as follows (with the rows operations performed top to bottom each time by using ${n \choose i}-{n-1 \choose i-1}={n-1 \choose i}$).
It is easy to see that
\begin{equation*}
A=A_{0,0}=\left(
  \begin{array}{cccccc}
    {{c+u_1-1 \choose 0}}+x & {{c+u_2-1 \choose 0}} & {{c+k-1 \choose 0}} & \cdots & {{c-1 \choose 0}}\\
    {{c+u_1 \choose 1}} & {{c+u_2 \choose 1}} & {{c+k \choose 1}} & \cdots & {{c \choose 1}}\\
    {\vdots} & {\vdots} & {\vdots} & {\vdots} & {\vdots}\\
    {{c+u_1 \choose \lambda}} & {{c+u_2 \choose \lambda}} & {{c+k \choose \lambda}} & \cdots & {{c \choose \lambda}}\\
  \end{array}
\right).
\end{equation*}
Multiplying the first row of $A_{0,0}$ by $-1$ and then adding it to the second row of $A_{0,0}$ yields $A_{0,1}$:
\begin{equation*}
A_{0,1}=\left(
  \begin{array}{cccccc}
    {{c+u_1-1 \choose 0}}+x & {{c+u_2-1 \choose 0}} & {{c+k-1 \choose 0}} & \cdots & {{c-1 \choose 0}}\\
    {{c+u_1-1 \choose 1}}-x & {{c+u_2-1 \choose 1}} & {{c+k-1 \choose 1}} & \cdots & {{c-1 \choose 1}}\\
    {\vdots} & {\vdots} & {\vdots} & {\vdots}& {\vdots}\\
    {{c+u_1 \choose \lambda}} & {{c+u_2 \choose \lambda}} & {{c+k \choose \lambda}} & \cdots & {{c \choose \lambda}}\\
  \end{array}
\right).
\end{equation*}
Repeat this process $\lambda$ times, i.e., for $1\leq i\leq \lambda$ multiply the $(i+1)$-th row of $A_{0,i}$ by $-1$ and then add it to the $(i+2)$-th row of $A_{0,i}$. It follows that
\begin{equation*}
A_{0,\lambda}=\left(
  \begin{array}{cccccc}
    {{c+u_1-1 \choose 0}}+(-1)^0x & {{c+u_2-1 \choose 0}} & {{c+k-1 \choose 0}} & \cdots & {{c-1 \choose 0}}\\
    {{c+u_1-1 \choose 1}}+(-1)^1x & {{c+u_2-1 \choose 1}} & {{c+k-1 \choose 1}} & \cdots & {{c-1 \choose 1}}\\
    {{c+u_1-1 \choose 2}}+(-1)^2x & {{c+u_2-1 \choose 2}} & {{c+k-1 \choose 2}} & \cdots & {{c-1 \choose 2}}\\
    {\vdots} & {\vdots} & {\vdots} & {\vdots}& {\vdots}\\
    {{c+u_1-1 \choose \lambda}}+(-1)^{\lambda}x & {{c+u_2-1 \choose \lambda}} & {{c+k-1 \choose \lambda}} & \cdots & {{c-1 \choose \lambda}}\\
  \end{array}
\right)
\end{equation*}
\begin{equation*}
=\left(
  \begin{array}{cccccc}
    {{c+u_1-2 \choose 0}}+(-1)^0x\times {{1 \choose 0}} & {{c+u_2-2 \choose 0}} & {{c+k-2 \choose 0}} & \cdots & {{c-2 \choose 0}}\\
    {{c+u_1-1 \choose 1}}+(-1)^1x\times {{1 \choose 1}} & {{c+u_2-1 \choose 1}} & {{c+k-1 \choose 1}} & \cdots & {{c-1 \choose 1}}\\
    {{c+u_1-1 \choose 2}}+(-1)^2x\times {{2 \choose 2}}& {{c+u_2-1 \choose 2}} & {{c+k-1 \choose 2}} & \cdots & {{c-1 \choose 2}}\\
    {\vdots} & {\vdots} & {\vdots} & {\vdots}& {\vdots}\\
    {{c+u_1-1 \choose \lambda}}+(-1)^{\lambda}x\times {{\lambda \choose \lambda}} & {{c+u_2-1 \choose \lambda}} & {{c+k-1 \choose \lambda}} & \cdots & {{c-1 \choose \lambda}}\\
  \end{array}
\right).
\end{equation*}
Multiplying the first row of $A_{0,\lambda}$ by $-1$ and then adding it to the second row of $A_{0,\lambda}$ yields $A_{1,1}$:
\begin{equation*}
A_{1,1}=\left(
  \begin{array}{ccccccc}
    {{c+u_1-2 \choose 0}}+(-1)^0x\times {{1 \choose 0}} & {{c+u_2-2 \choose 0}} & {{c+k-2 \choose 0}} & \cdots & {{c-2 \choose 0}}\\
    {{c+u_1-2 \choose 1}}+(-1)^1x\times {{2 \choose 1}} & {{c+u_2-2 \choose 1}} & {{c+k-2 \choose 1}} & \cdots & {{c-2 \choose 1}}\\
    {{c+u_1-1 \choose 2}}+(-1)^2x\times {{2 \choose 2}}& {{c+u_2-1 \choose 2}} & {{c+k-1 \choose 2}} & \cdots & {{c-1 \choose 2}}\\
    {\vdots} & {\vdots} & {\vdots} & {\vdots}& {\vdots}\\
    {{c+u_1-1 \choose \lambda}}+(-1)^{\lambda}x\times {{\lambda \choose \lambda}} & {{c+u_2-1 \choose \lambda}} & {{c+k-1 \choose \lambda}} & \cdots & {{c-1 \choose \lambda}}\\
  \end{array}
\right).
\end{equation*}
Repeat this process $\lambda$ times, i.e., for $1\leq i\leq \lambda$ multiply the $(i+1)$-th row of $A_{1,i}$ by $-1$ and then add it to the $(i+2)$-th row of $A_{1,i}$. It follows that
\begin{equation*}
A_{1,\lambda}=\left(
  \begin{array}{ccccccc}
    {{c+u_1-2 \choose 0}}+(-1)^0x\times {{2 \choose 0}} & {{c+u_2-2 \choose 0}} & {{c+k-2 \choose 0}} & \cdots & {{c-2 \choose 0}}\\
    {{c+u_1-2 \choose 1}}+(-1)^1x\times {{2 \choose 1}} & {{c+u_2-2 \choose 1}} & {{c+k-2 \choose 1}} & \cdots & {{c-2 \choose 1}}\\
    {{c+u_1-2 \choose 2}}+(-1)^2x\times {{3 \choose 2}}& {{c+u_2-2 \choose 2}} & {{c+k-2 \choose 2}} & \cdots & {{c-2 \choose 2}}\\
    {\vdots} & {\vdots} & {\vdots} & {\vdots}& {\vdots}\\
    {{c+u_1-2 \choose \lambda}}+(-1)^{\lambda}x\times {{\lambda +1 \choose \lambda}} & {{c+u_2-2 \choose \lambda}} & {{c+k-2 \choose \lambda}} & \cdots & {{c-2 \choose \lambda}}\\
  \end{array}
\right)
\end{equation*}
Again repeating the above technique on row transformations $l-1 \ (1\leq l\leq c)$ times one obtains
\begin{equation*}
A_{l-1,\lambda}=\left(
\begin{array}{cccccc}
    {{c+u_1-l \choose 0}}+(-1)^0x\times {{l \choose 0}} & {{c+u_2-l \choose 0}} & {{c+k-l \choose 0}} & \cdots & {{c-l \choose 0}}\\
    {{c+u_1-l \choose 1}}+(-1)^1x\times {{l \choose 1}} & {{c+u_2-l \choose 1}} & {{c+k-l \choose 1}} & \cdots & {{c-l \choose 1}}\\
    {{c+u_1-l \choose 2}}+(-1)^2x\times {{l+1 \choose 2}} & {{c+u_2-l \choose 2}} & {{c+k-l \choose 2}} & \cdots & {{c-l \choose 2}}\\
    {\vdots} & {\vdots} & {\vdots} & {\vdots}& {\vdots}\\
    {{c+u_1-l \choose \lambda}}+(-1)^{\lambda}x\times {{\lambda +l-1 \choose \lambda}} & {{c+u_2-l \choose \lambda}} & {{c+k-l \choose \lambda}} & \cdots & {{c-l \choose \lambda}}\\
  \end{array}
\right).
\end{equation*}
In particular,
\begin{equation*}
A':=A_{c-1,\lambda}=\left(
  \begin{array}{ccccc}
    {{u_1 \choose 0}}+(-1)^0x\times {{c \choose 0}} & {{u_2 \choose 0}} & {{k \choose 0}} & \cdots & {{0 \choose 0}}\\
    {{u_1 \choose 1}}+(-1)^1x\times {{c \choose 1}} & {{u_2 \choose 1}} & {{k \choose 1}} & \cdots & {{0 \choose 1}}\\
    {{u_1 \choose 2}}+(-1)^2x\times {{c+1 \choose 2}} & {{u_2 \choose 2}} & {{k \choose 2}} & \cdots & {{0 \choose 2}}\\
    {\vdots} & {\vdots} & {\vdots} & {\vdots}& {\vdots}\\
    {{u_1 \choose \lambda}}+(-1)^{\lambda}x\times {{\lambda+c-1 \choose \lambda}} & {{u_2 \choose \lambda}} & {{k \choose \lambda}} & \cdots & {{0 \choose \lambda}}\\
  \end{array}
\right).
\end{equation*}
We complete the proof.
\end{proof}

In the following, we will give some conditions for a zero-sum sequence $T$ over a $p$-group $G$ of length $|T|\geq 2k$ such that $T$ has a zero-sum subsequence of length $\leq k-1$ .

\begin{lem} \label{zerosub}
Let $p$ be a prime and $k$ be a positive integer.
Let $G$ be a finite abelian $p$-group and  $T$ be a zero-sum sequence over $G$ of length $|T|\geq 2k$.
Suppose that $2k\geq \mathsf{D}(G)+2$.
If there exists some $i\in [1,2k-\mathsf{D}(G)]$ such that
\begin{align*}
a_i:={{|T|-k \choose k-i}}+(-1)^{i}\times {{|T|-k+i-1 \choose k-1}}\not\equiv 0 \mod p,
\end{align*}
then $T$ has a zero-sum subsequence of length $\leq k-1$.
\end{lem}

Note that $a_1=0$.

\begin{proof}
Suppose that each zero-sum subsequence of $T$ has the length $\geq k$.
Since $\sigma(T)=0$, we must have that each zero-sum subsequence of $T$ has the length $\leq |T|-k$.
Hence, $$\mathsf{N}^i(T)=0\ \text{for}\ i\in [1,k-1]\cup[|T|-k+1,|T|-1].$$
It follows from Lemma \ref{lem1} that the following equations hold
\begin{equation}
\label{|T|}
\begin{aligned}
&1+(-1)^{k}\mathsf{N}^{k}(T)+ \ldots +(-1)^{|T|-k}\mathsf{N}^{|T|-k}(T)+(-1)^{|T|}\mathsf{N}^{|T|}(T)=\\
&(1+(-1)^{|T|})+(-1)^{k}\mathsf{N}^{k}(T)+ \ldots +(-1)^{|T|-k}\mathsf{N}^{|T|-k}(T)\equiv 0\mod p,
\end{aligned}
\end{equation}
and for any $T_1|T$ with $|T_1|\in [\mathsf{D}(G),|T|-1]$
$$1+(-1)^{k}\mathsf{N}^{k}(T_1)+ \ldots +(-1)^{|T|-k}\mathsf{N}^{|T|-k}(T_1)\equiv 0\mod p.$$
It follows that for $1\leq t\leq |T|-\mathsf{D}(G)$
$$\sum\limits _{T_1\mid T,\ |T_1|=|T|-t}\big(1+(-1)^{k}\mathsf{N}^{k}(T_1)+ \ldots +(-1)^{|T|-k}\mathsf{N}^{|T|-k}(T_1)\big)\equiv 0\mod p.$$
Analysing the number of times each subsequence is counted, one obtains
\begin{equation}
\label{Tequ}
\begin{aligned}
\binom{|T|}{|T_1|}&+(-1)^{k}\binom{|T|-k}{|T_1|-k}\mathsf{N}^{k}(T) \\
&+\ldots+(-1)^{|T|-k}\binom{|T|-(|T|-k)}{|T_1|-(|T|-k)}\mathsf{N}^{|T|-k}(T)\\
=\binom{|T|}{t}&+(-1)^{k}\binom{|T|-k}{t}\mathsf{N}^{k}(T) \\
&+\ldots+(-1)^{|T|-k}\binom{k}{t}\mathsf{N}^{|T|-k}(T)\equiv 0 \mod p .
\end{aligned}
\end{equation}
Set $X=(1,(-1)^{k}\mathsf{N}^{k}(T), \ldots, (-1)^{|T|-k}\mathsf{N}^{|T|-k}(T))^T$ and
\begin{equation*}
A:=\left(
  \begin{array}{cccc}
    {{|T| \choose 0}}+(-1)^{|T|} & {{|T|-k \choose 0}} & \cdots & {{k \choose 0}}\\
    {{|T| \choose 1}} & {{|T|-k \choose 1}} & \cdots & {{k \choose 1}}\\
    {{|T| \choose 2}} & {{|T|-k \choose 2}} & \cdots & {{k \choose 2}}\\
    {\vdots} & {\vdots} & {\vdots} & {\vdots}\\
    {{|T| \choose |T|-\mathsf{D}(G)}} & {{|T|-k \choose |T|-\mathsf{D}(G)}} & \cdots & {{k \choose |T|-\mathsf{D}(G)}}\\
  \end{array}
\right).
\end{equation*}
On the one hand, it can be deduced from (\ref{|T|}) and (\ref{Tequ}) that
$$AX\equiv 0\mod p.$$
Combining $2k>\mathsf{D}(G)$ with Lemma \ref{matrix0} yields that by some row transformations, $A$ can be brought to the following form
\begin{equation*}
A'=\left(
  \begin{array}{cccc}
    A_{11} & A_{12}\\
    A_{21} & A_{22}
  \end{array}
\right)
\end{equation*}
\begin{equation*}
=\left(
  \begin{array}{c:cccccc}
    {{|T|-k \choose 0}}+(-1)^0(-1)^{|T|}\times {{k \choose 0}} & {{|T|-2k \choose 0}} & \cdots & {{0 \choose 0}}\\
    {{|T|-k \choose 1}}+(-1)^1(-1)^{|T|}\times {{k \choose 1}} &{{|T|-2k \choose 1}} & \cdots & {{0 \choose 1}}\\
    {{|T|-k \choose 2}}+(-1)^2(-1)^{|T|}\times {{k+1 \choose 2}}& {{|T|-2k \choose 2}} & \cdots & {{0 \choose 2}}\\
    {\vdots} & {\vdots} & {\vdots} & {\vdots}\\
    {{|T|-k \choose |T|-2k}}+(-1)^{|T|-2k}(-1)^{|T|}\times {{|T|-k-1 \choose |T|-2k}}& {{|T|-2k \choose |T|-2k}} & \cdots & {{0 \choose |T|-2k}}\\
    \hdashline
     {{|T|-k \choose |T|-2k+1}}+(-1)^{|T|-2k+1}(-1)^{|T|}\times {{|T|-k \choose |T|-2k+1}}& {{|T|-2k \choose |T|-2k+1}} & \cdots & {{0 \choose |T|-2k+1}}\\
     {{|T|-k \choose |T|-2k+2}}+(-1)^{|T|-2k+2}(-1)^{|T|}\times {{|T|-k+1 \choose |T|-2k+2}}& {{|T|-2k \choose |T|-2k+2}} & \cdots & {{0 \choose |T|-2k+2}}\\
    {\vdots} & {\vdots} & {\vdots} & {\vdots}\\
    {{|T|-k \choose |T|-\mathsf{D}(G)}}+(-1)^{|T|-\mathsf{D}(G)}(-1)^{|T|}\times {{|T|-\mathsf{D}(G)+k-1 \choose |T|-\mathsf{D}(G)}} & {{|T|-2k \choose |T|-\mathsf{D}(G)}} & \cdots & {{0 \choose |T|-\mathsf{D}(G)}}\\
  \end{array}
\right),
\end{equation*}
where
\begin{equation*}
A_{21}=\left(
  \begin{array}{ccccc}
   {{|T|-k \choose |T|-2k+1}}+(-1)^{|T|-2k+1}(-1)^{|T|}\times {{|T|-k \choose |T|-2k+1}}\\
     {{|T|-k \choose |T|-2k+2}}+(-1)^{|T|-2k+2}(-1)^{|T|}\times {{|T|-k+1 \choose |T|-2k+2}}\\
    {\vdots}\\
    {{|T|-k \choose |T|-\mathsf{D}(G)}}+(-1)^{|T|-\mathsf{D}(G)}(-1)^{|T|}\times {{|T|-\mathsf{D}(G)+k-1 \choose |T|-\mathsf{D}(G)}}\\
  \end{array}
\right).
\end{equation*}
Obviously,
$$A'X\equiv 0\mod p.$$
For $i\in [1,2k-\mathsf{D}(G)]$, in $A_{21}$ we have that
\begin{align*}
&{{|T|-k \choose |T|-2k+i}}+(-1)^{|T|-2k+i}(-1)^{|T|}\times {{|T|-k+i-1 \choose |T|-2k+i}}\\
&={{|T|-k \choose k-i}}+(-1)^{i}\times {{|T|-k+i-1 \choose k-1}}=a_i.
\end{align*}
Since $a_i\not\equiv 0 \mod p$ for some $i\in [1,2k-\mathsf{D}(G)]$,
we have that the following equation holds
$$a_i+(-1)^{k}\binom{|T|-2k}{|T|-2k+i}\mathsf{N}^{k}(T)
+\ldots+(-1)^{|T|-k}\binom{0}{|T|-2k+i}\mathsf{N}^{|T|-k}(T)\equiv 0 \mod p.$$
Note that $\binom{j}{|T|-2k+i}=0$ if $j<|T|-2k+i$. Thus we must have $a_i\equiv 0 \mod p.$ This is contradiction.
We complete the proof.
\end{proof}

Fix $i_0\in [1,2k-\mathsf{D}(G)]$ be the smallest positive integer such that $a_{i_0}\not\equiv 0 \mod p$.
Since $a_1=0$, we have that $i_0\geq 2$.
In the following, we will simplify the condition that $a_i\not\equiv 0 \mod p$ for some $i\in [1,2k-\mathsf{D}(G)]$ in Lemma \ref{zerosub} by considering $i_0$.

We will use the following lemma repeatly.

\begin{lem} \label{bin}
Let $n,j$ be positive integers. Then
$$(-1)^j{{n+j-1 \choose j}}={{-n \choose j}}.$$
\end{lem}

In Lemma \ref{zerosub}, set $|T|-k=up+v$ and $k=cp+d$ with $v,d\in [0,p-1]$. We can get the followings.

\begin{lem} \label{simplify}
Suppose that $d\geq v+1$.
\begin{description}
  \item[(1)] If ${{u \choose c}}\not\equiv 0 \mod p$, then
  \begin{description}
    \item[(i)] $i_0=d-v$ if $2|d-v$;
    \item[(ii)] $i_0=d-v+1$ if $2\nmid d-v$ and $v+d\neq p$;
    \item[(iii)] $i_0=l_0p$, where $l_0$ is the smallest positive integer such that  ${{u \choose c-l_0}}+(-1)^{1+l_0}\times {{u+l_0 \choose c}}\not\equiv 0 \mod p$, if $2\nmid d-v$ and $v+d=p$.
  \end{description}
  \item[(2)] If ${{u \choose c}}\equiv 0 \mod p$, then $i_0\geq p+d-v$.
\end{description}
\end{lem}

Note that if $i<0$, then we specify that $a_i=0$.

\begin{proof}
By Lemma \ref{zerosub}, we have that for $i\in [2,2k-\mathsf{D}(G)]$
\begin{equation}
\label{a_i1}
\begin{aligned}
a_i&={{|T|-k \choose k-i}}+(-1)^{i}\times {{|T|-k+i-1 \choose k-1}}\\
&={{up+v \choose cp+d-i}}+(-1)^{i}\times {{up+v+i-1 \choose cp+d-1}}.
\end{aligned}
\end{equation}

For $i\in [2,d-v-1]$, we have that $p-1\geq d\geq d-i\geq d-(d-v-1)=v+1>v$ and $1\leq v+1\leq v+i-1\leq v+(d-v-1)-1=d-2<d-1$. Together this with Lemma \ref{luc} yields that
\begin{equation}
\label{a_if}
\begin{aligned}
a_i&={{up+v \choose cp+d-i}}+(-1)^{i}\times {{up+v+i-1 \choose cp+d-1}}\\
&\equiv {{u \choose c}}{{v \choose d-i}}+(-1)^{i}\times {{u \choose c}}{{v+i-1 \choose d-1}}\equiv 0 \mod p.
\end{aligned}
\end{equation}
Note that if $d\leq v+2$, then $[2,d-v-1]$ is empty and $a_i=0$ for $i\in [2,d-v-1]$.

Set $$i=d-v+j\ \text{with}\ j\geq 0.$$
For $i\in [d-v,\min\{d,p-v\}]$, we have that $j\in [0,\min\{v,p-d\}]$, $0\leq d-1\leq d-1+j\leq d-1+(p-d)=p-1$ and $0\leq v-j\leq v\leq p-1$.
Together this with Lemma \ref{luc} yields that
\begin{equation}
\label{a_i}
\begin{aligned}
a_i&={{up+v \choose cp+d-i}}+(-1)^{i}\times {{up+v+i-1 \choose cp+d-1}}\\
&={{up+v \choose cp+v-j}}+(-1)^{d-v+j}\times {{up+d+j-1 \choose cp+d-1}}\\
&\equiv {{u \choose c}}{{v \choose v-j}}+(-1)^{d-v+j}\times {{u \choose c}}{{d+j-1 \choose d-1}}\\
&= {{u \choose c}}\bigg({{v \choose j}}+(-1)^{d-v+j}\times {{d-1+j \choose j}}\bigg)\\
&= {{u \choose c}}\bigg({{v \choose j}}+(-1)^{d-v}\times {{-d \choose j}}\bigg) \mod p\ (\text{since Lemma}\ \ref{bin}).
\end{aligned}
\end{equation}

(1) If $2|d-v$, then since $d\geq v+1$, we have that $d\geq v+2$. This together with $d,v\in [0,p-1]$ yields that $p-1\geq d\geq v+2\geq 2$, which implies that $p\geq 3$.
By (\ref{a_i}) we have that
$$a_i\equiv {{u \choose c}}\bigg({{v \choose j}}+{{p-d \choose j}}\bigg) \mod p.$$
Take $j=0$ and we have $$a_i=a_{d-v}\equiv 2{{u \choose c}}\not\equiv 0 \mod p.$$
It follows since ${{u \choose c}}\not\equiv 0 \mod p$ and $p\geq 3$.
Hence, if ${{u \choose c}}\not\equiv 0 \mod p$ and $2|d-v$, then $i_0=d-v$, i.e, (i) holds.

If $2\nmid d-v$, then by (\ref{a_i}) we have that
$$a_i\equiv {{u \choose c}}\bigg({{v \choose j}}-{{p-d \choose j}}\bigg) \mod p.$$
Obviously, if $j=0$, then $a_i=a_{d-v}=0.$

If $d+v\neq p$, then since $d\geq v+1$ and $d,v\in [0,p-1]$, we must have that $d+v\not\equiv 0 \mod p$.
If $v=0$, then $2\nmid d-v=d$. Take $i=d-v+1=d+1$ and combining (\ref{a_i1}) with Lemma \ref{luc} yields that
\begin{align*}
a_i&={{up \choose cp+1}}+{{up+d \choose cp+d-1}}\equiv {{u \choose c}}{{0 \choose 1}}+{{u \choose c}}{{d \choose d-1}}\\
&=d{{u \choose c}}\not\equiv 0 \mod p.
\end{align*}
It follows since ${{u \choose c}}\not\equiv 0 \mod p$ and $d\in [0,p-1]$.
If $v\geq 1$, then $\min\{v,p-d\}\geq 1$. Take $j=1$ and we have that
\begin{align*}
a_i\equiv {{u \choose c}}\bigg({{v \choose 1}}-{{p-d \choose 1}}\bigg)\equiv {{u \choose c}}(v+d)\not\equiv 0 \mod p.
\end{align*}
It follows since ${{u \choose c}}\not\equiv 0 \mod p$ and $d+v\not\equiv 0 \mod p$.
Hence, if ${{u \choose c}}\not\equiv 0 \mod p$, $2\nmid d-v$ and $d+v\neq p$, then $i_0=d-v+1$, i.e., (ii) holds.

If $d+v=p$, then since $d\geq v+1$ and $d,v\in [0,p-1]$, we must have that
$p\geq 3$, otherwise, $p=2=d+v\geq 2v+1\geq 1$, which implies that $v=0$ and $d=2$, a contradiction to $d\in [0,p-1]$.
Set $$j=lp+j_p\ \text{with}\ l\geq 0 \ \text{and}\ j_p\in [0,p-1].$$
Since $2\nmid d-v$ and $p\geq 3$, it follows from (\ref{a_i}) that
\begin{equation}
\begin{aligned}
\label{a_i2}
a_i&={{up+v \choose cp+v-j}}+(-1)^{d-v+j}\times {{up+d+j-1 \choose cp+d-1}}\\
&={{up+v \choose (c-l)p+v-j_p}}+(-1)^{1+l+j_p}\times {{(u+l)p+d+j_p-1 \choose cp+d-1}}.
\end{aligned}
\end{equation}
If $v-j_p<0$, then since $d+v=p$, we have that $2p>d+j_p-1>d+v-1=p-1$.
Combining (\ref{a_i2}) with Lemma \ref{luc} yields that
\begin{align*}
a_i&={{up+v \choose (c-l-1)p+p+v-j_p}}+(-1)^{1+l+j_p}\times {{(u+l+1)p+d+j_p-1-p \choose cp+d-1}}\\
&\equiv {{u \choose c-l-1}}{{v \choose p+v-j_p}}+(-1)^{1+l+j_p}\times {{u+l+1 \choose c}}{{d+j_p-1-p \choose d-1}} \mod p.
\end{align*}
Obviously, $v<p+v-j_p$ and $d+j_p-1-p< d-1$.
It folllows that $a_i\equiv 0 \mod p.$
If $v-j_p\geq 0$, then since $d+v=p$, we have that $d-1\leq d+j_p-1\leq d+v-1=p-1$.
Combining (\ref{a_i2}) with Lemma \ref{luc} yields that
\begin{align*}
a_i&={{up+v \choose (c-l)p+v-j_p}}+(-1)^{1+l+j_p}\times {{(u+l)p+d+j_p-1 \choose cp+d-1}}\\
&\equiv {{u \choose c-l}}{{v \choose v-j_p}}+(-1)^{1+l+j_p}\times {{u+l \choose c}}{{d+j_p-1 \choose d-1}}\\
&= {{u \choose c-l}}{{v \choose j_p}}+(-1)^{1+l+j_p}\times {{u+l \choose c}}{{d+j_p-1 \choose j_p}}\\
&= {{u \choose c-l}}{{v \choose j_p}}+(-1)^{1+l}\times {{u+l \choose c}}{{-d \choose j_p}}\ (\text{since Lemma}\ \ref{bin})\\
&\equiv {{u \choose c-l}}{{v \choose j_p}}+(-1)^{1+l}\times {{u+l \choose c}}{{p-d \choose j_p}}\\
&= {{u \choose c-l}}{{v \choose j_p}}+(-1)^{1+l}\times {{u+l \choose c}}{{v \choose j_p}}\\
&= {{v \choose j_p}}\bigg({{u \choose c-l}}+(-1)^{1+l}\times {{u+l \choose c}}\bigg) \mod p.
\end{align*}
Since $p-1\geq v\geq j_p\geq 0$, we have that ${{v \choose j_p}}\not\equiv 0 \mod p$. It follows that $a_i\not\equiv 0 \mod p$ iff ${{u \choose c-l}}+(-1)^{1+l}\times {{u+l \choose c}}\not\equiv 0 \mod p$.
Obviously, $l\neq 0$, otherwise, ${{u \choose c-l}}+(-1)^{1+l}\times {{u+l \choose c}}=0$.
Hence, take $j_p=0$ and $i_0=l_0p$, where $l_0$ is the smallest positive integer such that  ${{u \choose c-l}}+(-1)^{1+l_0}\times {{u+l_0 \choose c}}\not\equiv 0 \mod p$, i.e., (iii) holds.

(2) Since  ${{u \choose c}}\equiv 0 \mod p$, by (\ref{a_if}) and (\ref{a_i}), we have that $a_i\equiv 0 \mod p$ for $i\in [2,\min\{d,p-v\}]$.
If $i=d-v+j\in [\min\{d+1,p-v+1\}, p+d-v-1]$, i.e., $j\in [\min\{v+1,p-d+1\},p-1]$, then either $v+1\leq j$ or $p-d+1\leq j$, i.e., either $v-j<0$ or $d+j-1\geq p$.

If $v-j<0$ and $d+j-1<p$, i.e., $j\in [v+1,p-d]$,  then $0\leq v< p+v-j\leq p-1$ and $0\leq d-1\leq d+j-1\leq p-1$.  This together with (\ref{a_i}) yields that
\begin{align*}
a_i&={{up+v \choose cp+v-j}}+(-1)^{d-v+j}\times {{up+d+j-1 \choose cp+d-1}}\\
&={{up+v \choose (c-1)p+p+v-j}}+(-1)^{d-v+j}\times {{up+d+j-1 \choose cp+d-1}}\\
&\equiv {{u \choose c-1}}{{v \choose p+v-j}}+(-1)^{d-v+j}\times {{u \choose c}}{{d+j-1 \choose d-1}}\\
&= (-1)^{d-v+j}\times {{u \choose c}}{{d+j-1 \choose d-1}}\equiv 0 \mod p\ (\text{since}\  {{u \choose c}}\equiv 0 \mod p).
\end{align*}

If $v-j<0$ and $d+j-1\geq p$, i.e., $j\in [\max\{v+1,p-d+1\},p-1]$,  then $0\leq v< p+v-j\leq p-1$ and $0\leq d+j-1-p< d-1\leq p-1$.  This together with (\ref{a_i}) yields that
\begin{align*}
a_i&={{up+v \choose cp+v-j}}+(-1)^{d-v+j}\times {{up+d+j-1 \choose cp+d-1}}\\
&={{up+v \choose (c-1)p+p+v-j}}+(-1)^{d-v+j}\times {{(u+1)p+d+j-1-p \choose cp+d-1}}\\
&\equiv {{u \choose c-1}}{{v \choose p+v-j}}+(-1)^{d-v+j}\times {{u+1 \choose c}}{{d+j-1-p \choose d-1}}=0 \mod p.
\end{align*}

If $v-j\geq 0$ and $d+j-1\geq p$, i.e., $j\in [p-d+1,v]$,  then $0\leq v-j\leq v\leq p-1$ and $0\leq d+j-1-p< d-1\leq p-1$.  This together with (\ref{a_i}) yields that
\begin{align*}
a_i&={{up+v \choose cp+v-j}}+(-1)^{d-v+j}\times {{up+d+j-1 \choose cp+d-1}}\\
&={{up+v \choose cp+v-j}}+(-1)^{d-v+j}\times {{(u+1)p+d+j-1-p \choose cp+d-1}}\\
&\equiv {{u \choose c}}{{v \choose v-j}}+(-1)^{d-v+j}\times {{u+1 \choose c}}{{d+j-1-p \choose d-1}}\\
&= {{u \choose c}}{{v \choose v-j}}\equiv 0 \mod p\ (\text{since}\  {{u \choose c}}\equiv 0 \mod p).
\end{align*}
Hence, $i_0\geq p+d-v$ and we complete the proof.
\end{proof}

If $d\leq v$, then the determination of $i_0$ is complex. In the following, we will focus on the cases $i=p+d-v$ or $k\equiv 1 \mod p$.

In Lemma \ref{leqlpcl} and Corollary \ref{leqlpclcor}, set $|T|-k=up+v=(u_1p^{t}+u_2)p+v$ and $k=cp+d=c_1p^{t+1}+d$ with $v,d\in [0,p-1]$, $t\geq 0$, $c_1,u_1\in[1,p-1]$ and $u_2\in [0,p^{t}-1]$.
It follows that $$u=u_1p^{t}+u_2\ \text{and}\ c=c_1p^t.$$
Note that if $t=0$, then $u=u_1$, $u_2=0$ and $c=c_1$.
Since $|T|\geq 2k$, we have that $|T|-k=u_1p^{t+1}+u_2p+v\geq k=c_1p^{t+1}+d$, which implies that
\begin{align}
\label{*}
p-1\geq u_1\geq c_1\geq 1.
\end{align}

\begin{lem} \label{leqlpcl}
Suppose that $2k-\mathsf{D}(G)\geq p+d-v$.
If $${{u_1 \choose c_1-1}}+(-1)^{p+d-v}\times {{u_1+1 \choose c_1}}\not\equiv 0 \mod p,$$
then $i_0\leq p+d-v$.
\end{lem}

\begin{proof}
By the definition of $i_0$, it suffices to prove that $$a_{p+d-v}\not\equiv 0 \mod p.$$
Since $2k-\mathsf{D}(G)\geq p+d-v$, by the definition (\ref{a_i1}) of $a_i$, we have that
\begin{align*}
a_{p+d-v}&={{up+v \choose (c-1)p+v}}+(-1)^{p+d-v}\times {{(u+1)p+d-1 \choose cp+d-1}}\\
&\equiv {{u \choose c-1}}{{v \choose v}}+(-1)^{p+d-v}\times {{u+1 \choose c}}{{d-1 \choose d-1}}\\
&={{u_1p^{t}+u_2 \choose c_1p^t-1}}+(-1)^{p+d-v}\times {{u_1p^{t}+u_2+1 \choose c_1p^t}} \mod p.
\end{align*}
If $t=0$, our result obviously holds.

Now suppose that $t\geq 1$.
If $u_2\in [0,p^{t}-2]$, then it follows from the above that
\begin{align*}
a_{p+d-v}&\equiv {{u_1p^{t}+u_2 \choose c_1p^t-1}}+(-1)^{p+d-v}\times {{u_1p^{t}+u_2+1 \choose c_1p^t}}\\
&\equiv {{u_1 \choose c_1-1}}{{u_2 \choose p^t-1}}+(-1)^{p+d-v}\times {{u_1 \choose c_1}}{{u_2+1 \choose 0}}\\
&=(-1)^{p+d-v}\times {{u_1 \choose c_1}} \mod p.
\end{align*}
Since $p-1\geq u_1\geq c_1\geq 1$, we have that $a_{p+d-v}\equiv (-1)^{p+d-v}\times {{u_1 \choose c_1}}\not\equiv 0 \mod p.$
If $u_2=p^{t}-1$, then it follows from the above that
\begin{align*}
a_{p+d-v}&\equiv {{u_1p^{t}+u_2 \choose c_1p^t-1}}+(-1)^{p+d-v}\times {{u_1p^{t}+u_2+1 \choose c_1p^t}}\\
&={{(u_1+1)p^{t}-1 \choose c_1p^t-1}}+(-1)^{p+d-v}\times {{(u_1+1)p^{t} \choose c_1p^t}}\\
&\equiv {{u_1 \choose c_1-1}}{{p^t-1 \choose p^t-1}}+(-1)^{p+d-v}\times {{u_1+1 \choose c_1}}{{0 \choose 0}}\\
&={{u_1 \choose c_1-1}}+(-1)^{p+d-v}\times {{u_1+1 \choose c_1}}\not\equiv 0 \mod p.
\end{align*}
The last inequality follows from the lemma's assumed condition.
We complete the proof.
\end{proof}

\begin{cor} \label{leqlpclcor}
Suppose that $2k-\mathsf{D}(G)\geq p+d-v$.
If $u_1+c_1+1<p,$
then $i_0\leq p+d-v$.
\end{cor}

\begin{proof}
By Lemma \ref{leqlpcl}, it suffices to prove that ${{u_1 \choose c_1-1}}+(-1)^{p+d-v}\times {{u_1+1 \choose c_1}}\not\equiv 0 \mod p.$
If $2\nmid p+d-v$, then
$${{u_1 \choose c_1-1}}+(-1)^{p+d-v}\times {{u_1+1 \choose c_1}}={{u_1 \choose c_1-1}}-{{u_1+1 \choose c_1}}=-{{u_1 \choose c_1}}\not\equiv 0 \mod p,$$
where ${{u_1 \choose c_1}}\not\equiv 0 \mod p$ follows since $p-1\geq u_1\geq c_1\geq 1$.
If $2\mid p+d-v$, then
Obviously,
\begin{align*}
&{{u_1 \choose c_1-1}}+(-1)^{p+d-v}\times {{u_1+1 \choose c_1}}\\
&={{u_1 \choose c_1-1}}+{{u_1 \choose c_1-1}}\times \frac{u_1+1}{c_1}\\
&={{u_1 \choose c_1-1}}\bigg(1+\frac{u_1+1}{c_1}\bigg)\\
&={{u_1 \choose c_1-1}}\bigg(\frac{u_1+c_1+1}{c_1}\bigg)\not\equiv 0 \mod p,
\end{align*}
where ${{u_1 \choose c_1-1}}\not\equiv 0 \mod p$ follows since $p-1\geq u_1\geq c_1\geq 1$, and $\frac{u_1+c_1+1}{c_1}\not\equiv 0 \mod p$ follows since $u_1+c_1+1<p$ and $p-1\geq u_1\geq c_1\geq 1$.
We complete the proof.
\end{proof}

In the following, we consider the case $k\equiv 1 \mod p.$
\begin{lem} \label{k1}
Set $|T|-k=u_1p^{t}+v_1$ and $k=c_1p^{t}+1$ with $t\geq 1$, $u_1,c_1\in [1,p-1]$ and $v_1\in [0,p^t-1]$.
Suppose that $2k-\mathsf{D}(G)\geq 2$.
If $${{u_1 \choose c_1-1}}+{{u_1+1 \choose c_1}}\not\equiv 0 \mod p,$$
then $i_0=2$.
\end{lem}

\begin{proof}
By the definition of $i_0$, it suffices to prove that $$a_{2}\not\equiv 0 \mod p.$$
Since $|T|\geq 2k$, we have that $|T|-k=u_1p^{t}+v_1\geq k=c_1p^{t}+1$, which implies that $p-1\geq u_1\geq c_1\geq 1$.
Since
\begin{equation*}
\begin{aligned}
a_2&={{|T|-k \choose k-2}}+(-1)^{2}\times {{|T|-k+1 \choose k-1}}\\
&={{u_1p^{t}+v_1 \choose c_1p^{t}-1}}+{{u_1p^{t}+v_1+1 \choose c_1p^{t}}}.
\end{aligned}
\end{equation*}
If $v_1\in [0,p^{t}-2]$, then $v_1+1\in [1,p^{t}-1]$.
This together with Lemma \ref{luc} yields that
\begin{align*}
a_{2}&={{u_1p^{t}+v_1 \choose c_1p^{t}-1}}+{{u_1p^{t}+v_1+1 \choose c_1p^{t}}}\\
&\equiv {{u_1 \choose c_1-1}}{{v_1 \choose p^t-1}}+{{u_1 \choose c_1}}{{v_1+1 \choose 0}}\\
&={{u_1 \choose c_1}}\not\equiv 0 \mod p,
\end{align*}
where ${{u_1 \choose c_1}}\not\equiv 0 \mod p$ follows from $p-1\geq u_1\geq c_1\geq 1$.
If $v_1=p^{t}-1$, then by Lemma \ref{luc} we have that
\begin{align*}
a_{2}&\equiv {{u_1p^{t}+v_1 \choose c_1p^{t}-1}}+{{u_1p^{t}+v_1+1 \choose c_1p^{t}}}\\
&={{(u_1+1)p^{t}-1 \choose c_1p^t-1}}+{{(u_1+1)p^{t} \choose c_1p^t}}\\
&\equiv {{u_1 \choose c_1-1}}{{p^t-1 \choose p^t-1}}+{{u_1+1 \choose c_1}}{{0 \choose 0}}\\
&={{u_1 \choose c_1-1}}+{{u_1+1 \choose c_1}}\not\equiv 0 \mod p.
\end{align*}
The last inequality follows from the lemma's assumed condition.
We complete the proof.
\end{proof}

\section{The proofs of Theorems \ref{leqlp} and \ref{leqG}}

In this section, we give some conditions such that $\mathsf{s}_{\leq k-1}(G)\leq 2\mathsf{D}(G)-k+1$ holds, and prove Theorems \ref{leqlp} and \ref{leqG}.

\begin{lem} \label{pleqk-1}
Let $G$ be a finite abelian $p$-group and  $k\in [\exp(G)+1,\mathsf{D}(G)]$ be a positive integer.
Let $S$ be a sequence over $G$ of length $2\mathsf{D}(G)-k+1$.
Suppose that $\mathsf{N}^i(S)=0$ for $i\in [\mathsf{D}(G)+1,|S|]$.
If $\binom{\mathsf{D}(G)}{k-1}\not\equiv 0 \mod p$, then $S$ has a nonempty zero-sum subsequence $T$ with $|T|\leq k-1$.
\end{lem}

\begin{proof}
Suppose that each zero-sum subsequence of $S$ has the length $\geq k$.
Since $\mathsf{N}^i(S)=0$ for $i\in [\mathsf{D}(G)+1,|S|]$,
by Lemma \ref{lem1} we have that the following equation holds for any $T|S$ with $|T|\in [\mathsf{D}(G),|S|]$
$$1+(-1)^{k}\mathsf{N}^{k}(T)+ \ldots +(-1)^{\mathsf{D}(G)}\mathsf{N}^{\mathsf{D}(G)}(T)\equiv 0\mod p.$$
It follows that for $0\leq t\leq |S|-\mathsf{D}(G)$
$$\sum\limits _{T\mid S,\ |T|=|S|-t}\big(1+(-1)^{k}\mathsf{N}^{k}(T)+ \ldots +(-1)^{\mathsf{D}(G)}\mathsf{N}^{\mathsf{D}(G)}(T)\big)\equiv 0\mod p.$$
Analysing the number of times each subsequence is counted, one obtains
\begin{equation}
\label{SD(G)}
\begin{aligned}
\binom{|S|}{|T|}&+(-1)^{k}\binom{|S|-k}{|T|-k}\mathsf{N}^{k}(S) \\
&+\ldots+(-1)^{\mathsf{D}(G)}\binom{|S|-\mathsf{D}(G)}{|T|-\mathsf{D}(G)}\mathsf{N}^{\mathsf{D}(G)}(S)\\
=\binom{|S|}{t}&+(-1)^{k}\binom{|S|-k}{t}\mathsf{N}^{k}(S) \\
&+\ldots+(-1)^{\mathsf{D}(G)}\binom{|S|-\mathsf{D}(G)}{t}\mathsf{N}^{\mathsf{D}(G)}(S)\equiv 0 \mod p .
\end{aligned}
\end{equation}
Set $X=(1,(-1)^{k}\mathsf{N}^{k}(S), \ldots, (-1)^{\mathsf{D}(G)}\mathsf{N}^{\mathsf{D}(G)}(S))^T$ and
\begin{equation*}
A:=\left(
  \begin{array}{cccc}
    {{|S| \choose 0}} & {{|S|-k \choose 0}} & \cdots & {{|S|-\mathsf{D}(G) \choose 0}}\\
    {{|S| \choose 1}} & {{|S|-k \choose 1}} & \cdots & {{|S|-\mathsf{D}(G) \choose 1}}\\
    {\vdots} & {\vdots} & {\vdots} & {\vdots}\\
    {{|S| \choose |S|-\mathsf{D}(G)}} & {{|S|-k \choose |S|-\mathsf{D}(G)}} & \cdots & {{|S|-\mathsf{D}(G) \choose |S|-\mathsf{D}(G)}}\\
  \end{array}
\right)_{(|S|-\mathsf{D}(G)+1)\times (\mathsf{D}(G)-k+2)}.
\end{equation*}
On the one hand, it can be deduced from (\ref{SD(G)}) that
$$AX\equiv 0\mod p.$$
It follows from Lemma \ref{matrix0} that by some row transformations, $A$ can be brought to the following form
\begin{equation*}
A'=\left(
  \begin{array}{cccc}
    {{\mathsf D(G) \choose 0}} & {{\mathsf{D}(G)-k \choose 0}} & \cdots & {{0 \choose 0}}\\
    {{\mathsf D(G) \choose 1}} & {{\mathsf{D}(G)-k \choose 1}} & \cdots & {{0 \choose 1}}\\
    {\vdots} & {\vdots} & {\vdots} & {\vdots}\\
    {{\mathsf D(G) \choose |S|-\mathsf{D}(G)}} & {{\mathsf{D}(G)-k \choose |S|-\mathsf{D}(G)}} & \cdots & {{0 \choose |S|-\mathsf{D}(G)}}\\
  \end{array}
\right)_{(|S|-\mathsf{D}(G)+1)\times (\mathsf{D}(G)-k+2)}
\end{equation*}
and $$AX\equiv A'X\equiv 0 \mod p.$$
Since $|S|=2\mathsf{D}(G)-k+1$, we have that $|S|-\mathsf{D}(G)+1=\mathsf{D}(G)-k+2$. It follows that $A'$ is a square matrix and $|S|-\mathsf{D}(G)>\mathsf{D}(G)-k$.
Note that $\binom{a}{b}=0$ if $0\leq a<b$. We have that
$${{\mathsf{D}(G)-k \choose |S|-\mathsf{D}(G)}}= \cdots = {{0 \choose |S|-\mathsf{D}(G)}}=0.$$
Thus
\begin{equation*}
\left(
  \begin{array}{cccc}
   {{\mathsf D(G) \choose |S|-\mathsf{D}(G)}} & {{\mathsf{D}(G)-k \choose |S|-\mathsf{D}(G)}} & \cdots & {{0 \choose |S|-\mathsf{D}(G)}}\\
  \end{array}
\right)X={{\mathsf D(G) \choose |S|-\mathsf{D}(G)}}=\binom{\mathsf{D}(G)}{k-1}\equiv 0\ \text{mod} \ p.
\end{equation*}
This is a contradiction to $\binom{\mathsf{D}(G)}{k-1}\not\equiv 0 \mod p$.
We complete the proof.
\end{proof}

By Corollary \ref{leqlpclcor}, we prove Theorem \ref{leqlp}.

\textit{The proof of Theorem \ref{leqlp}:}
If $\mathsf{N}^i(S)=0$ for $i\in [\mathsf{D}(G)+1,|S|]$, then by Lemma \ref{pleqk-1} we are done.

Now suppose that $S$ has a zero-sum subsequence $T$ with length $|T|\in [\mathsf{D}(G)+1,|S|]$.
If $|T|<2k$, then since $T$ is a zero-sum subsequence of $S$ with length $|T|\geq \mathsf{D}(G)+1$, by the definition of $\mathsf{D}(G)$ we have that $T$ has proper zero-sum subsequences $T_1, TT_1^{-1}$. Obviously, $\min\{|T_1|,|TT_1^{-1}|\}\leq k-1$. We are done.

Suppose that $|T|\geq 2k$. As in Lemma \ref{leqlpcl} we set $|T|-k=up+v=(u_1p^{t}+u_2)p+v$ and $k=cp+d=c_1p^{t+1}+d$ with $v,d\in [0,p-1]$, $t\geq 0$, $c_1,u_1\in[1,p-1]$ and $u_2\in [0,p^{t}-1]$. Repeat the reasoning of (\ref{*}) and we have that
$$u=u_1p^{t}+u_2,\ c=c_1p^t\ \text{and}\ p-1\geq u_1\geq c_1\geq 1.$$
Since $2k-\mathsf{D}(G)\geq p+d-v$, by Lemma \ref{zerosub} and Corollary \ref{leqlpclcor}, it suffices to prove that $$u_1+c_1+1<p.$$
Since $|S|=2\mathsf{D}(G)-k+1$ and $2\mathsf{D}(G)-2k+1<\frac{p-1}{2}p^{t+1}$, we have that
\begin{align*}
|T|-k=(u_1p^{t}+u_2)p+v=u_1p^{t+1}+u_2p+v\leq |S|-k=2\mathsf{D}(G)-2k+1<\frac{p-1}{2}p^{t+1}.
\end{align*}
It follows that $u_1<\frac{p-1}{2}$. This together with $p-1\geq u_1\geq c_1\geq 1$ yields that $$u_1+c_1+1\leq 2u_1+1<p.$$
We complete the proof.
\qed

By Lemma \ref{k1}, we prove Theorem \ref{leqG}.

\textit{The proof of Theorem \ref{leqG}:}
Let $S$ be a sequence over $G$ of length $|S|=2\mathsf{D}(G)-k+1$. In the following, we will prove that $S$ has a nonempty zero-sum subsequence of length $\leq k-1$. If there exists some $i\in [\mathsf{D}(G)+1,2k-1]$ such that $\mathsf{N}^i(S)\neq 0$, then $S$ has a zero-sum subsequence $T$ with length $|T|\in [\mathsf{D}(G)+1,2k-1]$. By the definition of $\mathsf{D}(G)$, $T$ has proper zero-sum subsequences $T_1, TT_1^{-1}$. Obviously, $\min\{|T_1|,|TT_1^{-1}|\}\leq k-1$. We are done.
Hence, we can suppose that
\begin{align}
\label{Ni}
\mathsf{N}^i(S)=0\ \text{for}\ i\in [\mathsf{D}(G)+1,2k-1].
\end{align}

(i) Since $\mathsf{D}(C_2^r)=r+1=2^{t+1}-1$ and $k-1=\frac{r+2}{2}=2^t$, by Lemma \ref{luc}, we have that $$\binom{\mathsf{D}(C_2^r)}{k-1}=\binom{2^{t+1}-1}{2^t}\equiv \binom{1}{1}\binom{2^{t}-1}{0}=1\not\equiv 0 \mod 2.$$
If $\mathsf{N}^i(S)=0$ for $i\in [\mathsf{D}(C_2^r)+1,|S|]=[2^{t+1}, 2^{t+1}+2^t-2]$, then by Lemma \ref{pleqk-1} we are done.

If there exists some $i\in [\mathsf{D}(C_2^r)+1,|S|]=[2^{t+1}, 2^{t+1}+2^t-2]$ such that $\mathsf{N}^i(S)\neq 0$, then
since $2k=2^{t+1}+2>\mathsf{D}(C_2^r)+1$, combining (\ref{Ni}) yields that $S$ has a zero-sum subsequence $T$ with length $|T|\in [2^{t+1}+2, 2^{t+1}+2^t-2]$. Thus $|T|\geq 2^{t+1}+2=2k$ and $|T|-k\in [2^{t}+1, 2^{t+1}-3]$.
Set $|T|-k=u_12^{t}+v_1=2^{t}+v_1$ and $k=c_12^{t}+1=2^t+1$ with $u_1=c_1=1$ and $v_1\in [1,2^t-3]$.
Since $2k-\mathsf{D}(C_2^r)=3\geq 2$ and $${{u_1 \choose c_1-1}}+{{u_1+1 \choose c_1}}={{1 \choose 1-1}}+{{1+1 \choose 1}}=1\not\equiv 0 \mod 2,$$ by Lemma \ref{k1} we have that $a_2\not\equiv 0 \mod 2.$
It follows from Lemma \ref{zerosub} that $S$ has a zero-sum subsequence of length $\leq k-1$.
We complete the proof of (i).

(ii)   Since $\mathsf{D}(C_p^4)=4p-3$ and $k-1=2p$, by Lemma \ref{luc}, we have that $$\binom{\mathsf{D}(C_p^4)}{k-1}=\binom{4p-3}{2p}\equiv \binom{3}{2}\binom{p-3}{0}=3\not\equiv 0 \mod p.$$ It follows since $p\geq 5$.
If $\mathsf{N}^i(S)=0$ for $i\in [\mathsf{D}(C_p^4)+1,|S|]=[4p-3, 6p-6]$, then by Lemma \ref{pleqk-1} we are done.
If there exists some $i\in [\mathsf{D}(C_p^4)+1,|S|]=[4p-3, 6p-6]$ such that $\mathsf{N}^i(S)\neq 0$, then
since $2k=4p+2>\mathsf{D}(C_p^4)+1$, combining (\ref{Ni}) yields that $S$ has a zero-sum subsequence $T$ with length $|T|\in [4p+2, 6p-6]$. Thus $|T|\geq 4p+2=2k$ and $|T|-k\in [2p+1,4p-7]$.

If $3p-1\not\in [2p+1,4p-7]$, i.e., $p=5$, then  $|T|-k\in [2p+1,3p-2]$.
Set $|T|-k=2p+v_1$ with $v_1\in [1,p-2]$.
Thus by Lemma \ref{luc} we have that
\begin{align*}
a_2&={{|T|-k \choose k-2}}+{{|T|-k+1 \choose k-1}}={{2p+v_1 \choose 2p-1}}+{{2p+v_1+1 \choose 2p}}\\
&\equiv \binom{2}{1}\binom{v_1}{p-1}+\binom{2}{2}\binom{v_1+1}{0}=1\not\equiv 0 \mod p.
\end{align*}

If $3p-1\in [2p+1,4p-7]$, i.e., $p\geq 6$, then set $|T|-k=u_1p+v_1$ with $u_1=2$ or $3$ and $v_1\in [0,p-1]$. If $|T|-k\neq 3p-1$, then $v_1\in [0.p-2]$.
This together with Lemma \ref{luc} yields that
\begin{align*}
a_2&={{|T|-k \choose k-2}}+{{|T|-k+1 \choose k-1}}={{u_1p+v_1 \choose 2p-1}}+{{u_1p+v_1+1 \choose 2p}}\\
&\equiv \binom{u_1}{1}\binom{v_1}{p-1}+\binom{u_1}{2}\binom{v_1+1}{0}
=\binom{u_1}{2}\not\equiv 0 \mod p.
\end{align*}
If $|T|-k=3p-1$, then
\begin{align*}
a_2&={{|T|-k \choose k-2}}+{{|T|-k+1 \choose k-1}}={{3p-1 \choose 2p-1}}+{{3p \choose 2p}}\\
&\equiv \binom{2}{1}\binom{p-1}{p-1}+\binom{3}{2}\binom{0}{0}=5\not\equiv 0 \mod p.
\end{align*}
It follows since $p\geq 6$.
Since $2k-\mathsf{D}(C_p^4)=5\geq 2$, by Lemma \ref{zerosub} $S$ has a zero-sum subsequence of length $\leq k-1$.
We complete the proof of (ii).

(iii)   Since $\mathsf{D}(C_p^d)=(p-1)d+1$ and $k-1=(d-1)p\in [p,\mathsf{D}(C_p^d)]$, we have that $(p-1)d+1\geq (d-1)p\geq p$, i.e., $2\leq d\leq p+1$. It follows that $p-d+1\in [0,p-1]$. This together with Lemma \ref{luc} yields that $$\binom{\mathsf{D}(C_p^d)}{k-1}=\binom{(d-1)p+p-d+1}{(d-1)p}\equiv \binom{d-1}{d-1}\binom{p-d+1}{0}=1\not\equiv 0 \mod p.$$
If $\mathsf{N}^i(S)=0$ for $i\in [\mathsf{D}(C_p^d)+1,|S|]=[(p-1)d+2, (d+1)p-2d+2]$, then by Lemma \ref{pleqk-1}, we are done.
If there exists some $i\in [\mathsf{D}(C_p^d)+1,|S|]=[(p-1)d+2, (d+1)p-2d+2]$ such that $\mathsf{N}^i(S)\neq 0$, then
since $2k-(\mathsf{D}(C_p^d)+1)=2(d-1)p+2-(p-1)d-2=(d-2)p+d>0$ (since $d\geq 2$), combining (\ref{Ni}) yields that $S$ has a zero-sum subsequence $T$ with length $|T|\in [2(d-1)p+2, (d+1)p-2d+2]$.
It follows that $2(d-1)p+2\leq (d+1)p-2d+2$.
It is also easy to see that $2(d-1)p+2-((d+1)p-2d+2)=(d-3)p+2d>0$ if $d\geq 3$. Hence, $d=2$.
Note that $\mathsf{s}_{\leq p}(C_p^2)=3p-2=2(\mathsf{D}(C_p^2))-p$.
We complete the proof.
\qed

\section{Concluding Remarks and Open Problems}

We close the paper by making the following conjectures together with some remarks on $k_G$ and $\mathsf{s}_{k \exp(G)}(G)$ for finite abelian groups $G$.

By Theorem \ref{D(G)-2}, we have the following conjecture.

\begin{conj}\label{D(G)-1}
Let $G$ be a finite abelian group with $r(G)\geq 2$ and $\mathsf{D}(G)=\mathsf{D}^*(G)$. Suppose that $G\not\cong C_2^4$ and $\mathsf{D}(G)-2\geq \exp(G)$. If $\exp(G)<\frac{\mathsf{D}(G)-1}{2}$,
then $\mathsf{s}_{\leq \mathsf{D}(G)-2}(G)=\mathsf{D}(G)+1$.
\end{conj}

By Theorems \ref{leqlp} and \ref{leqG}, we have the following conjecture.

\begin{conj}\label{D(G)-1/2}
Let $G$ be a finite abelian group with $r(G)\geq 2$ and $\mathsf{D}(G)=\mathsf{D}^*(G)$. Let $k$ be a positive integer with $\mathsf{D}(G)-k\in [\exp(G),\mathsf{D}(G)]$. If $\mathsf{D}(G)-k\geq \frac{\mathsf{D}(G)+1}{2}$,
then $$\mathsf{s}_{\leq \mathsf{D}(G)-k}(G)\leq \mathsf{D}(G)+k.$$
If $\mathsf{D}(G)-k< \frac{\mathsf{D}(G)+1}{2}$,
then $$\mathsf{s}_{\leq \mathsf{D}(G)-k}(G)> \mathsf{D}(G)+k.$$
\end{conj}
It follows from  Lemma \ref{col} Parts (i) and (ii) that Conjectures \ref{D(G)-1} and  \ref{D(G)-1/2} holds for $G=C_3^3$ and $C_5^3$. Schmid and Zhuang \cite{SZ10} proved that the first assertion holds if $\mathsf{D}(G)-k=\exp(G)$ and $G$ is a finite abelian $p$-group.

Gao et al. \cite{[GHPS]} conjectured that

\begin{conj}\label{kDG}
Let $G$ be a finite abelian group. If $k \exp(G)\geq \mathsf{D}(G)$, then
$$\mathsf{s}_{k \exp(G)}(G) = \mathsf{D}(G) + k \exp(G) - 1.$$
\end{conj}

Gao and Thangadurai \cite{GT} noticed that if $k \exp(G) < \mathsf{D}(G)$, then
$\mathsf{s}_{k \exp(G)}(G) > k \exp(G) + \mathsf{D}(G)-1$.
If Conjecture \ref{D(G)-1/2} holds, then
$$\mathsf{s}_{\leq k \exp(G)}(G)\leq 2\mathsf{D}(G)-k \exp(G) \ \text{if}\ k \exp(G)\geq \frac{\mathsf{D}(G)+1}{2}$$
and $$\mathsf{s}_{\leq k \exp(G)}(G)> 2\mathsf{D}(G)-k \exp(G) \ \text{if}\ k \exp(G)< \frac{\mathsf{D}(G)+1}{2}.$$

By the above and Conjectures \ref{Gao}, \ref{kDG}, we suggest the following conjecture.

\begin{conj}\label{kexpG}
Let $G$ be a finite abelian group. If $k \exp(G)\in [\frac{\mathsf{D}(G)+1}{2},\mathsf{D}(G)]$, then
$$\mathsf{s}_{k \exp(G)}(G)\leq 2\mathsf{D}(G) - 1.$$
If $k \exp(G)<\frac{\mathsf{D}(G)+1}{2}$, then
$$\mathsf{s}_{k \exp(G)}(G)> 2\mathsf{D}(G) - 1.$$
\end{conj}

Gao and Thangadurai \cite{GT} proved that $\mathsf{s}_{2\cdot 3}(C_3^3)=13=2\mathsf{D}(C_3^3) - 1$, which satisfies Conjecture \ref{kexpG}. Note that Conjectures \ref{Gao} and \ref{kexpG} are not always true, but we think that these conjectures still make sense.
Sidorenko \cite{S20} proved that $\mathsf{s}_{2m}(C_2^{2m+1})=4m+5$ for odd $m$.
Wang and Zhao \cite{WZ17} proved that $\mathsf{s}_{\leq r-k}(C_2^r)=r+2\ \text{for all}\
r-k\in \Big[\Big\lceil \frac{2r+2}{3}\Big\rceil, r\Big]$.
Thus $\mathsf{s}_{\leq 2m}(C_2^{2m+1})=2m+3\ \text{for}\ m\geq 2$.
It follows that
$$\mathsf{s}_{2m}(C_2^{2m+1})=4m+5> \mathsf{s}_{\leq 2m}(C_2^{2m+1})+2m-1=4m+2 \ \text{for} \ m\geq 3 \ \text{odd}.$$
Hence, Conjecture \ref{Gao} is not always true.
Since $\mathsf{D}(C_2^{2m+1})=2m+2$, we have that $2m\in [\frac{\mathsf{D}(C_2^{2m+1})+1}{2},\mathsf{D}(C_2^{2m+1})]=[m+2,2m+2]$ for $m\geq 2$.
It follows that $$\mathsf{s}_{2m}(C_2^{2m+1})=4m+5>2\mathsf{D}(C_2^{2m+1}) - 1=4m+3 \ \text{for} \ m\geq 3 \ \text{odd}.$$
Hence, Conjecture \ref{kexpG} is not always true.


\bibliographystyle{amsalpha}

\end{document}